\documentclass[11pt]{amsart}
\usepackage{amsmath}
\usepackage{amsthm}
\usepackage{amsfonts}
\usepackage{amssymb}
\newtheorem{Theorem}{Theorem}[section]
\newtheorem{Proposition}[Theorem]{Proposition}
\newtheorem{Lemma}[Theorem]{Lemma}
\newtheorem{Corollary}[Theorem]{Corollary}
\theoremstyle{definition}

\theoremstyle{remark}
\newtheorem{remark}{Remark}

\numberwithin{equation}{section}

\newcommand{\R}{{\mathbb R}}
\newcommand{\C}{{\mathbb C}}
\newcommand{\N}{{\mathbb N}}

\newcommand{\tr}{{\textrm{\rm tr}\:}}
\renewcommand{\Im}{{\textrm{\rm Im}\:}}

\begin{document}

\title{Canonical systems and quantum graphs}

\author{Kyle Scarbrough}

\address{Department of Mathematics\\ Faculty of Nuclear Sciences and Physical Engineering \\ Czech Technical University in Prague\\ Trojanova 13\\ 120 00 Prague\\ Czech Republic}
\email{kyle.d.scarbrough-1@ou.edu}

\date{\today}



\begin{abstract}
The representation of the resolvent as an integral operator, the $m$ function, and the associated spectral representation are fundamental topics in the spectral theory of self-adjoint ordinary differential operators. Versions of these are developed here for canonical systems $Ju'=-zHu$ of arbitrary order. Next, canonical systems on graphs, not necessarily compact but with finitely many vertices, are introduced and proved to be unitarily equivalent to certain higher order canonical systems. It is shown that any Schr\"odinger operator on a graph is unitarily equivalent to a canonical system on the same graph. Consequently, for an arbitrary canonical system or Schr\"odinger operator on a graph, a representation of the resolvent as an integral operator and a spectral representation are obtained.         
\end{abstract}
\maketitle

\section{Introduction}
A \textit{canonical system} is a differential equation of the form
\begin{equation}
\label{can}
Ju'(x) = -zH(x)u(x) , \quad J=\begin{pmatrix} 0 & -I \\ I & 0 \end{pmatrix} ,
\end{equation}
with $H(x)\in\C^{2n\times 2n}$ and $H(x)\ge 0$. These are considered here on intervals $(a,b)$ that can be bounded or unbounded. See Section 2 for more precise assumptions and definitions.  

Canonical systems of order two, meaning $n=1$, are particularly well-researched \cite{BHS}, \cite{Rembook}. Take such a system, and suppose, for example, that $(a,b)=(0,\infty)$, $H(x)\in\R^{2\times 2}$, $\tr H(x)=1$, and the boundary condition $u_2(0)=0$ for solutions of \eqref{can} is imposed. This canonical system, more precisely the relation generated by the differential equation together with the boundary condition, is then self-adjoint. It is known that the resolvent can be written as a certain integral involving the Green function. The Green function is built out of two solutions of \eqref{can}, one satisfying the boundary condition at $0$ and one in the underlying Hilbert space. The initial value of the latter solution gives rise to the Weyl $m$ function. The $m$ function is a scalar Herglotz function, an analytic function from the upper half plane $\C^+$ to $\overline{\C^+}$, and such functions have the representation
\[
m(z) = A + Bz + \int_{-\infty}^{\infty} \left( \frac{1}{t-z} - \frac{t}{t^2+1} \right)\, d\rho(t) ,
\]
with $A\in\R \cup \{\infty\}$, $B\ge 0$, and $\rho$ a Borel measure on $\R$ with $\int\frac{d\rho(t)}{1+t^2}<\infty$. The canonical system has a spectral representation $U$ in $L^2(\rho)$ given by the formula $(Uf)(t)= \lim_{L\to \infty}\int_{0}^L u(x,t)^*H(x)f(x) \, dx$, where $u(x,t)$ is the solution of $Ju'=-tHu$ with $u(0,t)=\begin{pmatrix}1 \\ 0 \end{pmatrix}$. Moreover, such canonical systems are in one-to-one correspondence with arbitrary Herglotz functions $\C^+ \to \overline{\C^+}$. See \cite{Rembook} for more details.

Important steps in the basic direct spectral theory, as above, of higher order canonical systems were taken by Hinton, Shaw, and Krall in \cite{HS1}, \cite{HS2}, and \cite{Krall}. Those works make a more restrictive assumption about the canonical system, definiteness on every subinterval of $(a,b)$, than is usual in the general order two case. The above results mentioned, in the case $n=1$, concerning the resolvent and the spectral representation are obtained in this paper for arbitrary $n$ under the milder assumption that the canonical system is definite on $(a,b)$. See Section 2 for the precise notion of definiteness used here.  The methodology behind the proofs here is different as well. Hinton, Shaw, and Krall relied on a Weyl theory of nested disks. The approach in this paper, for the half line case in Section 4, is more closely modeled on that taken in \cite{Teschl} and \cite{WMLN}. The idea is to show that any spectral representation of a canonical system must be a kind of generalized eigenfunction expansion. The generalized eigenfunctions are then written in terms of a fixed basis of solutions satisfying the boundary condition. An explicit spectral representation, involving that fixed basis of solutions and the matrix valued measure associated with the system's $m$ function, is obtained by studying the resolvent and the $m$ function. This methodology apparently is new in the context of canonical systems. 

Other fundamental results, frequently used here, concerning higher order canonical systems are contained in \cite{BHSW} and \cite{Lesch}. A number of results concerning both direct and inverse spectral theory for higher order systems are discussed in \cite{Sakh}.

Quantum graphs are a major area of research in mathematical physics and spectral theory \cite{BCFK}, \cite {BK}, \cite{EKKST}, \cite {Post}. The basic object of study is a graph, a set of differential equations along the edges, and interface conditions at the vertices. The focus in this paper is on canonical systems and Schr\"odinger equations on graphs with finitely many vertices. The edges of the graphs are allowed to be half lines. See Sections 5, 6, and 7 for the precise definitions. 

Canonical systems on graphs are not as widely studied as Schr\"odinger equations, but see \cite{BLT} and \cite{dSW} for some results about canonical systems on star graphs. De Snoo and Winkler in \cite{dSW} obtained a spectral representation of a canonical system on a star graph; a key step in their argument was to study the resolvent using Krein's formula. It is not clear how to extend their method to more general graphs. More basically, it has not even been proved that the relation corresponding to a canonical system on a general graph with seemingly self-adjoint interface conditions is in fact self-adjoint.

The representation of the resolvent as an integral operator, $m$ functions, and spectral representations for Schr\"odinger operators on graphs have appeared in many specialized situations \cite{AP}, \cite{ASVCdL}, \cite{AK}, \cite{BL}, \cite{BER}, \cite{BW}, \cite{CW1}--\cite{CW3}, \cite{KS}, \cite{KN}, \cite{Roh}, \cite{SW}, \cite{Yur}. In \cite{LSS}, spectral representations as generalized eigenfunction expansions are considered in some generality for Schr\"odinger operators on graphs; the spectral measure is taken as an input in their result, the potentials are assumed to be locally $L^2$, and the generalized eigenfunctions are delivered by abstract methods that do not determine the specific initial values of the generalized eigenfunctions. Note that this is in contrast to the above results for canonical systems, where the spectral measure is derived from an $m$ function, the coefficient function in the differential equation is only assumed to be locally $L^1$, and the initial value of the solution providing the spectral representation is fixed.

A natural question is whether the resolvent and a spectral representation of a quantum graph can be given in explicit, universal forms, as is the case for the traditional self-adjoint ordinary differential and difference operators and relations. This question is answered here in two steps. The first step is to prove that the resolvent and a spectral representation of any higher order canonical system, on an interval, can be given in the expected forms, as mentioned above. The second step is to set up an explicit unitary equivalence between an arbitrary quantum graph and a higher order canonical system. By the general inverse spectral theory of canonical systems of order two, arbitrary spectral data in the form of scalar Herglotz functions can be realized by canonical systems. Moreover, well-known ordinary differential and difference equations, such as Schr\"odinger, Dirac, and Jacobi, on intervals can be rewritten explicitly as canonical systems \cite{Rembook}. Thus, the second step in answering the above question also answers another natural question, namely, whether a quantum graph, say with canonical systems or Schr\"odinger equations along the edges, can be rewritten as a canonical system on an interval. 

In Sections 5 and 6, it is shown that a canonical system on a graph with $k$ vertices is unitarily equivalent to a canonical system of order $4k$ on an interval. In particular, self-adjointness, a representation of the resolvent as an integral operator, and a spectral representation are obtained. For a compact graph, the main trick is to insert a vertex into the middle of each edge with Neumann-Kirchhoff interface conditions and to change the variables so that each of the canonical systems on the new edges has the same domain for $x$. This is similar to the well-known trick for turning a whole line problem into a higher order half line problem. It is also closely related to the developments in \cite{CW1}--\cite{CW3} for Sturm-Liouville operators on compact graphs. For the non-compact case, the fact that boundary conditions can be implemented by singular half lines, Theorem \ref{singhalfline}, is essentially the other tool needed. Note that this theorem does not have an analogue for Sturm-Liouville problems, so whether a Sturm-Liouville operator on a non-compact graph can be rewritten as a Sturm-Liouville system on a half line is an open problem.

However, it is proved in Section 7 that a Schr\"odinger operator on a graph is unitarily equivalent to a canonical system on that graph and, hence, to a higher order canonical system on a bounded interval or half line. This provides general forms for the representation of the resolvent as an integral operator and for a spectral representation. 

Here is an overview of the following sections. Basic notation, definitions, and results concerning linear relations and canonical systems are provided in Section 2. Sections 3 and 4 concern $m$ functions, the resolvents, and spectral representations for canonical systems on bounded intervals and half lines, respectively. Canonical systems on compact graphs and their unitary equivalence to higher order canonical systems on bounded intervals are discussed in Section 5. Canonical systems on non-compact graphs are examined in Section 6. In Section 7, Schr\"odinger operators on graphs are reviewed, and a unitary equivalence between an arbitrary Schr\"odinger operator on a graph and a canonical system on the graph is established.

\section{Preliminaries on relations and canonical systems}
Let $\mathcal H_1$ and $\mathcal H_2$  be Hilbert spaces. A \textit{relation} is a linear subspace $\mathcal R \subset \mathcal H_1 \oplus \mathcal H_2$. The \textit{domain} $D(\mathcal R)$ of a relation is defined as the set $\{f: (f,g)\in \mathcal R \textrm{ for some }g\in \mathcal H_2\}$. The \textit{multivalued part} of $\mathcal R$ is the set $\mathcal R(0)=\{g: (0,g)\in \mathcal R\}$. The \textit{inverse} and \textit{adjoint} of $\mathcal R$ are the relations
\[
\mathcal R^{-1}=\{ (g,f): (f,g)\in \mathcal R \}
\]
and
\[
\mathcal R^*=\{ (h,k)\in \mathcal H_2 \oplus \mathcal H_1 : \langle h,g\rangle =\langle k,f\rangle \textrm{ for all }(f,g)\in \mathcal R \} ,
\]
respectively, on $\mathcal H_2 \oplus \mathcal H_1$. If $\mathcal H_1=\mathcal H_2$, then $\mathcal R$ is called \textit{symmetric} if $\mathcal R \subset \mathcal R^*$ and \textit{self-adjoint} if $\mathcal R = \mathcal R^*$. If $\mathcal S \subset \mathcal H_2 \oplus \mathcal H_3$ is another relation, then the product is defined by $\mathcal S \mathcal R = \{(f,h) : (f,g) \in \mathcal R \textrm{ and } (g,h) \in \mathcal S \textrm{ for some } g\in \mathcal H_2 \}$. In particular, after identifying an operator with its graph, this defines the product of an operator and a relation.

Let $\mathcal S$ be a self-adjoint relation on $\mathcal H \oplus \mathcal H$, and define the relations $\mathcal S_1=\mathcal S \cap\left(\overline{D(\mathcal S)}\oplus \overline{D(\mathcal S)} \right)$ and $\mathcal S_2=\{(0,g): g\in \mathcal S(0)\}$. Then $\mathcal S=\mathcal S_1\oplus \mathcal S_2$, $\mathcal S_1$ is the graph of a self-adjoint operator in $\overline{D(\mathcal S)}$, and $D(\mathcal S)^\perp=\mathcal S(0)$. The operator whose graph is $\mathcal S_1$ is denoted by $S$. The \textit{resolvent} of $\mathcal S$ is $(\mathcal S-z)^{-1}$ for $z\in \C/\sigma(S)$. The resolvent is (the graph of) a bounded normal operator, and its kernel is $\mathcal S(0)$. A \textit{spectral representation} of $\mathcal S$ is a linear map $U:\mathcal H \to \mathcal L$ such that $\mathcal L$ is either the Hilbert space $\bigoplus L^2(\rho_j)$ for some Borel measures $\rho_j$ on $\R$ or the Hilbert space $L^2(\rho)$ for some nonnegative matrix valued Borel measure on $\R$, $U$ maps $\overline{D(\mathcal S)}$ isometrically onto $\mathcal L$, $\ker U=\mathcal S(0)$, and $U\mathcal S U^*=M_t$, the operator of multiplication by the variable in $\mathcal L$. 

Turning to canonical systems $Ju'=-zHu$, let $H(x)\in\C^{2n\times 2n}$ for $x\in (a,b)$ be given. The endpoints satisfy $-\infty \le a <b \le \infty$. It is assumed that $H(x)\ge 0$ for almost every $x\in(a,b)$ and that $H\in L_{\textrm{loc}}^1(a,b)$. An endpoint $a$ or $b$ is called \textit{regular} if $H\in L^1(a,c)$ or $H\in L^1(c,b)$, respectively, for some $c\in(a,b)$. $H$ is called \textit{regular} if both $a$ and $b$ are regular.  

$H$ is called \textit{definite} on $(c,d)\subset (a,b)$ if, for every $v\in \C^{2n}$, $H(x)v=0$ for almost every $x\in (c,d)$ implies $v=0$. This is equivalent to $0<\int_c^d u(x)^*H(x)u(x) \, dx\le \infty$ for every non-trivial solution $u$ of \eqref{can} \cite[Lemma 2.10]{BHSW}. It is known that $H$ is definite on $(a,b)$ if and only if there is a bounded interval $(c,d)\subset (a,b)$ such that $H$ is definite on $(c,d)$ \cite[Proposition 2.11]{BHSW}.

Suppose that $f$ and $g$ are solutions of $Jf'=-zHf$ and $Jg'=-wHg$, respectively. Note that solutions are assumed to be (locally) absolutely continuous. A quick calculation shows that $(f^*Jg)'=(\overline z -w)f^*Hg$. Hence, the Lagrange type identity
\begin{equation}
f(d)^*Jg(d)-f(c)^*Jg(c)=(\overline z -w)\int_c^{d} f(x)^*H(x)g(x)  \, dx 
\end{equation}
holds for any $(c,d)\subset(a,b)$ on which $H$ is regular.

Suppose $V$ is the vector space of all Borel measurable $f:(a,b) \to {\C}^{2n}$ such that $\int_a^{b} f(x)^*H(x)f(x)  \, dx < \infty$, together with the seminorm 
\[
\|f\|= \left(\int_a^{b} f(x)^*H(x)f(x)  \, dx \right)^{1/2} .
\]
Let $N = \{ f \in V : \|f\|=0 \}$, and define $L_H^2(a,b)=V/N$. $L_H^2(a,b)$ is a Hilbert space.

The \textit{maximal relation} of the canonical system $H$ is the relation
\begin{align}
\label{maxrel}
\mathcal{T} = \{ &(f,g) \in L_H^2(a,b) \oplus L_H^2(a,b) : f \textrm{ has an AC representative } \nonumber \\
 &f_0 \textrm{ such that } Jf_0'(x)=-Hg(x) \textrm{ for a.e. } x\in (a,b) \} 
\end{align}
in $L_H^2(a,b) \oplus L_H^2(a,b)$, and the \textit{minimal relation} is the relation $\mathcal T_0=\mathcal T^*$. Suppose $H$ is definite on $(a,b)$. Then, for every $(f,g)\in \mathcal T$, there is an absolutely continuous function $f_0$ that is uniquely determined by the properties that it represents $f$ in $L_H^2(a,b)$ and that $Jf_0'=-Hg$ almost everywhere on $(a,b)$ \cite[Proposition 2.15]{Lesch}. The notation $f_0$ is frequently used in this paper for this function determined by $(f,g)\in \mathcal T$.

The \textit{deficiency indices} of $H$ are the dimensions of $\ker(\mathcal T-i)$ and $\ker(\mathcal T+i)$. If $H$ is definite on $(a,b)$, then $H$ being regular is equivalent to both deficiency indices being equal to $2n$ \cite[Proposition 2.19, Theorem 5.14]{Lesch}

Suppose that $H$ is definite on $(a,b)$. Let $\alpha \in \C^{n\times 2n}$ and $\beta \in \C^{n\times 2n}$ be such that
\begin{equation}
\label{bc}
\alpha \alpha^*=I=\beta \beta^* , \quad \alpha J \alpha^*=0=\beta J \beta^* .
\end{equation}
If $H$ is regular, then
\begin{equation}
\label{regcs}
\mathcal{S}^{\alpha,\beta} = \{ (f,g) \in \mathcal{T} : \alpha f_0(a)=0=\beta f_0(b) \}
\end{equation}
is a self-adjoint relation in $L_H^2(a,b) \oplus L_H^2(a,b)$ \cite[Corollary 5.7]{BHSW}. The corresponding self-adjoint operator in $\overline{D(\mathcal{S}^{\alpha,\beta})}$ is denoted by $S^{\alpha,\beta}$. If $a$ is a regular endpoint and $H$ has deficiency indices both equal to $n$, then
\begin{equation}
\label{lpcs}
\mathcal{S}^{\alpha}=\mathcal S = \{ (f,g) \in \mathcal{T} : \alpha f_0(a)=0 \}
\end{equation}
is a self-adjoint relation \cite[Corollary 5.12]{BHSW}, and the corresponding self-adjoint operator is denoted by $S^{\alpha}$ or $S$ if $\alpha$ is clear from the context. 
     
\section{Regular canonical systems}
The goal in this section is to define the $m$ function, study the resolvent, and set up a spectral representation for a regular canonical system. Assume that $H$ is definite and regular on a bounded interval $(a,b)$, and that boundary conditions $\alpha \in \C^{n\times 2n}$ and $\beta \in \C^{n\times 2n}$ are fixed at $a$ and $b$, respectively, with the properties given in \eqref{bc}. Throughout this section, $u$ and $v$ denote the solutions of $\eqref{can}$ with 
\[
u(a,z)=-J\alpha^* , \quad v(a,z)=\alpha^*.
\]
Note that the matrix solution $u(x,z)$ satisfies the boundary condition $\alpha$ at $0$, and its columns form a basis for vector solutions of $\eqref{can}$ that satisfy that boundary condition.
\begin{Theorem}
\label{regm}
For every $z\in \C$ that is not an eigenvalue of the self-adjoint relation $\mathcal S^{\alpha, \beta}$ defined by \eqref{regcs}, there is unique matrix $m(z) \in \C^{n\times n}$ such that the solution $f_m(x,z)=v(x,z)+u(x,z)m(z)$ of \eqref{can} satisfies the boundary condition $\beta$ at $b$. This unique $m(z)$ is equal to $-(\beta u(b,z))^{-1}\beta v(b,z)$. Moreover, $m(z)$ is meromorphic on $\C$, the poles of $m(z)$ are all of the eigenvalues of $\mathcal S^{\alpha, \beta}$, $\Im(m(z))>0$ for $z\in \C^+$, and $m(\overline z)=m(z)^*$.
\end{Theorem}
\begin{proof} 
I claim that $z$ is an eigenvalue of $\mathcal S^{\alpha, \beta}$ if and only if $\beta u(b,z)$ has a non-trivial kernel. To see this, suppose first that $c \in \ker(\beta u(b,z))$ and $c\neq 0$. Notice that then $u(x,z)c$ is a solution of \eqref{can} that satisfies the boundary conditions $\alpha$ and $\beta$; since $c\neq 0$, it is not the trivial solution because $u(a,z)c \neq 0$. Since $H$ is definite on $(a,b)$, this means that $u(x,z)c$ is an eigenvector for $\mathcal S^{\alpha, \beta}$ with eigenvalue $z$. For the converse, suppose $z$ is an eigenvalue of $\mathcal S^{\alpha, \beta}$. So, there is a nontrivial solution of \eqref{can} that satisfies the boundary conditions $\alpha$ at $a$ and $\beta$ at $b$. Since this nontrivial solution satisfies the boundary condition $\alpha$, it is of the form $u(x,z)c$ for some nonzero $c \in \C^n$. Fix such a $c$. Since the solution $u(x,z)c$ also satisfies the boundary condition $\beta$, this means $c \in \ker(\beta u(b,z))$.

Now, assume that $z$ is not an eigenvalue of $\mathcal S^{\alpha, \beta}$. Thus, since $\beta u(b,z)$ has a trivial kernel, $m(z)=-(\beta u(b,z))^{-1}\beta v(b,z)$ is well-defined. It is trivial to check that $f_m(x,z)=v(x,z)+u(x,z)m(z)$ is a solution of \eqref{can} that satisfies the boundary condition $\beta$ at $b$. To see that $m(z)$ is unique, suppose that $m \in \C^{n\times n}$ is such that $v(x,z)+u(x,z)m$ satisfies the boundary condition $\beta$ at $b$. So, $0=\beta v(b,z)+\beta u(b,z)m$. Thus, since $\beta u(b,z)$ is invertible, $m=-(\beta u(b,z))^{-1}\beta v(b,z)$.

Since the initial values $u(a,z)$ and $v(a,z)$ are constant functions of $z$, the solutions $u(x,z)$ and $v(x,z)$ are entire functions of $z$. So, since $\beta u(b,z)$ is invertible if and only if $z$ is not an eigenvalue of $\mathcal S^{\alpha, \beta}$, the function $m(z)=-(\beta u(b,z))^{-1}\beta v(b,z)$ is meromorphic on $\C$ and its poles are all of the eigenvalues of $\mathcal S^{\alpha, \beta}$.

To show $\Im(m(z))>0$ for $z\in \C^+$, fix $z \in \C^+$. A direct calculation, using the fact that $f_m$ solves \eqref{can}, shows that $(f_m^*Jf_m)'=-2i\Im(z)f_m^*Hf_m$. So, 
{%
\small
\[
f_m(b,z)^*Jf_m(b,z)-f_m(a,z)^*Jf_m(a,z)=-2i\Im(z) \int_a^b f_m(x,z)^*H(x)f_m(x,z) \, dx .
\]
}%
Since $f_m$ satisfies the boundary condition $\beta$ at $b$, $f_m(b,z)^*Jf_m(b,z)=0$. A direct calculation, using the assumptions \eqref{bc} about $\alpha$, shows that $f(a,z)^*Jf(a,z)=m(z)-m(z)^*$. So, 
\[
\Im(m(z))=\Im(z) \int_a^b f_m(x,z)^*H(x)f_m(x,z) \, dx >0
\]
for $z\in \C^+$ since $H$ is definite on $(a,b)$.

To see that $m(\overline z)=m(z)^*$, first note that since $(f_m(x,z)^*Jf_m(x,\overline z))'=0$, $f_m(x,z)^*Jf_m(x,\overline z)$ is a constant function of $x$. Again, since $f_m$ satisfies the boundary condition $\beta$ at $b$, $f_m(b,z)^*Jf_m(b,\overline z)=0$. So, $0=f_m(a,z)^*Jf_m(a,\overline z)=m(\overline z)-m(z)^*$.
\end{proof}

\begin{Theorem}
\label{regres}
Define
\[
G(x,y,z) = \begin{cases} u(x,z)f_m(y,\overline z)^* & x\leq y \\ f_m(x,z)u(y,\overline z)^* & x>y \end{cases}
\]
for $z \in \C / \R$. Then 
\[
((\mathcal S^{\alpha, \beta}-z)^{-1}h)(x)= \int_a^b G(x,y,z)H(y)h(y) \, dy
\]
for all $h\in L_H^2(a,b)$.
\end{Theorem}
\begin{proof}
Fix $h\in L_H^2(a,b)$. Let $g(x)=\int_a^b G(x,y,z)H(y)h(y) \, dy$. We will show that the function $g(x)$ is a solution of $Jg'=-zHg-Hh$ that satisfies both boundary conditions $\alpha$ and $\beta$, which will imply that $(g,h) \in \mathcal S^{\alpha, \beta}-z$. Since 
\begin{align*}
\int_a^b G(x,y,z)H(y)h(y) \, dy=&f_m(x,z)\int_a^x u(y,\overline z)^*H(y)h(y) \, dy\\
&+u(x,z)\int_x^b f_m(y,\overline z)^*H(y)h(y) \, dy ,
\end{align*}
it is obvious that $g(x)=\int_a^b G(x,y,z)H(y)h(y) \, dy$ is an absolutely continuous function for $h\in L_H^2(a,b)$. By direct calculation, $g$ satisfies 
\[
Jg'(x)=-zH(x)g(x)+J(f_m(x,z)u(x,\overline z)^*-u(x,z)f_m(x,\overline z)^*)H(x)h(x) .
\]
Let $W(x,z)= (u(x,z) \: v(x,z))$. Then 
\[
f_m(x,z)u(x,\overline z)^*=W(x,z)\begin{pmatrix} m(z) & 0 \\ I & 0 \end{pmatrix}W(x,\overline z)^*
\]
and
\begin{align*}
u(x,z)f_m(x,\overline z)^*&=W(x,z)\begin{pmatrix} m(\overline z)^* & I \\ 0 & 0 \end{pmatrix}W(x,\overline z)^*\\
&=W(x,z)\begin{pmatrix} m(z) & I \\ 0 & 0 \end{pmatrix}W(x,\overline z)^* .
\end{align*}
So, $f_m(x,z)u(x,\overline z)^*-u(x,z)f_m(x,\overline z)^*=W(x,z)^*JW(x,\overline z)$. Now, $W$ is a solution of \eqref{can}, so $(W(x,z)^*JW(x,\overline z))'=0$. Hence, 
\begin{align*}
f_m(x,z)u(x,\overline z)^*-u(x,z)f_m(x,\overline z)^*&=W(a,z)^*JW(a,\overline z)\\
&=\begin{pmatrix} \alpha J \\ \alpha \end{pmatrix} \begin{pmatrix} \alpha^* & J\alpha^* \end{pmatrix}\\
&=J .
\end{align*}
Thus, $Jg'=-zHg-Hh$.

Since $(a,b)$ is finite and $g$ is absolutely continuous, $g\in L_H^2(a,b)$. So, $(g,zg+h) \in \mathcal T^{\alpha, \beta}$. Also, 
\[
\alpha g(a)=\alpha u(a,z)\int_a^b f_m(y,\overline z)^*H(y)h(y) \, dy =0
\]
and
\[
\beta g(b)=\beta f_m(b,z)\int_a^b u(y,\overline z)^*H(y)h(y) \, dy =0
\]
since $\alpha u(a,z)=0=\beta f_m(b,z)$. Hence, $(g,zg+h) \in \mathcal S^{\alpha, \beta}$. So, $(g,h) \in \mathcal S^{\alpha, \beta}-z$, and the theorem follows since $(\mathcal S^{\alpha, \beta}-z)^{-1}$ is an operator.
\end{proof}

\begin{Corollary}
\label{regHS}
$(\mathcal S^{\alpha, \beta}-z)^{-1}$ is a Hilbert-Schmidt operator for $z \in \C / \R$. Hence, $\sigma(\mathcal S^{\alpha, \beta}) = \{t_j\}$ is purely discrete and $\sum_j \frac{1}{1+t_j^2}<\infty$.
\end{Corollary}
\begin{proof}
Define the isometry $V: L_H^2(a,b) \to L^2(a,b)$ by $Vh=H^{1/2}h$. Then $V$ transforms $(\mathcal S^{\alpha, \beta}-z)^{-1}$ an integral operator on $V(L_H^2(a,b))$ with kernel 
\[
\begin{cases} H^{1/2}(x)u(x,z)f_m(y,\overline z)^*H^{1/2}(y) & x\leq y \\ H^{1/2}(x)f_m(x,z)u(y,\overline z)^*H^{1/2}(y) & x>y \end{cases} ,
\] 
which is clearly square integrable.
\end{proof} 
By Theorem \ref{regm}, $m(z)$ is a Herglotz function, i.e. a holomorphic function $\C^+ \to \C^{n\times n}$ such that $\Im(m(z))\ge 0$. It is well-known that such functions have a unique representation
\begin{equation}
\label{Herrep}
m(z) = A + Bz + \int_{-\infty}^{\infty} \left( \frac{1}{t-z} - \frac{t}{t^2+1} \right)\, d\rho(t) ,
\end{equation}
with $A=A^*$, $B\ge 0$, and $\rho$ a nonnegative matrix valued Borel measure on $\R$ such that $\int\frac{d\rho(t)}{1+t^2}<\infty$ \cite{GT}. This representation is called the Herglotz representation of $m(z)$. Note that $\rho$ is a discrete measure for $m$ coming from a regular canonical system by Theorem \ref{regm}.
\begin{Theorem}
\label{regrep}
Let
\[
(Uh)(t)=\int_a^b u(x,t)^*H(x)h(x) \, dx
\]
for $h\in L_H^2(a,b)$. Then $U:L_H^2(a,b) \to L^2(\rho)$, where $\rho$ is the measure in \eqref{Herrep}, provides a spectral representation of $\mathcal S^{\alpha, \beta}$. Moreover, $\rho$ can be reconstructed as follows. Let $\varphi_{jk}$, $k=1,\, \dots,\,M(j)$, be orthonormal eigenvectors corresponding to the eigenvalue $t_j\in \sigma(\mathcal S^{\alpha, \beta})$ with multiplicity $M(j)$. Choose $\varphi_{jk}$ to be solutions of $J\varphi'=-t_jH\varphi$, and write $\varphi_{jk}(x)=u(x,t_j)c_{jk}$. Then $\rho(\{t_j\})=\sum_{k=1}^{M(j)}c_{jk}c_{jk}^*$.
\end{Theorem}
\begin{proof}
The overall structure of the proof is as follows. First, it will be shown that $U$ is isometric on $\overline{D(\mathcal S^{\alpha, \beta})}$ and $\ker U= \mathcal S^{\alpha, \beta}(0)$ by computing $\lim_{y\to 0^+}-iy\langle h,(S^{\alpha, \beta}-z)^{-1}h \rangle$ for $z=t+iy \in \C^+$ and $h \in L_H^2(a,b)$ in two different ways, using the functional calculus and then using Theorem \ref{regres}. The formula for $\rho(\{t_j\})$ in the theorem statement will follow from similar computations with $\lim_{y\to 0^+}-iy(S^{\alpha, \beta}-z)^{-1}h$ for $z=t_j+iy \in \C^+$. Surjectivity of $U$ will then follow easily from the previous steps. Finally, straightforward computation and inspection of domains will show that $U\mathcal S^{\alpha, \beta}U^*=M_t$, where $M_t$ is multiplication by the independent variable $t$ in $L^2(\rho)$. Key throughout will be the fact that the spectrum of $\mathcal S^{\alpha, \beta}$ is purely discrete by Corollary \ref{regHS}.

Let $z=t+iy \in \C^+$. Since the spectrum is purely discrete,
\begin{equation}
\label{fc}
\lim_{y\to 0^+}-iy\langle h,(S^{\alpha, \beta}-z)^{-1}h \rangle = \begin{cases} \sum_{k=1}^{M(j)} |\langle \varphi_{jk},h \rangle |^2 & t=t_j\in \sigma(S^{\alpha, \beta}) \\ 0 & t \notin \sigma(S^{\alpha, \beta}) \end{cases}
\end{equation}
for all $h \in L_H^2(a,b)$ by the functional calculus. Also, by Theorem \ref{regres}, 
\[
\langle h,(S^{\alpha, \beta}-z)^{-1}h \rangle = \int_a^b \int_a^b h(x)^*H(x)G(x,s,z)H(s)h(s) \, ds\, dx .
\]
Let $W=(u \: v)$. Then
\[
f(x,z)u(s,\overline z)^*=W(x,z)\begin{pmatrix} m(z) & 0 \\ I & 0 \end{pmatrix}W(s,\overline z)^*
\]
and
\[
u(x,z)f(s,\overline z)^*=W(x,z)\begin{pmatrix} m(z) & I \\ 0 & 0 \end{pmatrix}W(s,\overline z)^* .
\]
So, 
\[
G(x,s,z)=W(x,z)\left( \begin{pmatrix} m(z) & \frac{1}{2}I \\ \frac{1}{2}I & 0 \end{pmatrix} +\frac{\sigma(x,s)}{2}J \right) W(s,\overline z)^* ,
\]
where 
\[
\sigma(x,s)=\begin{cases} -1 & x \leq s\\ 1 & x>s \end{cases} .
\]
Now, $W(x,t+iy)$ and $W(s,t-iy)$ are continuous functions for $x,s \in (a,b)$ and $y\in [0,\delta]$ for any $\delta>0$, and $Hh \in L^1(a,b)$. So, 
\[
-iy\int_a^b \int_a^b h(x)^*H(x)W(x,z)\left( \frac{\sigma(x,s)}{2}J \right) W(s,\overline z)^*H(s)h(s) \, ds\, dx 
 \]
 goes to $0$ as $y\to 0^+$, and, using the well-know fact that $\lim_{y\to 0^+}-iym(z)= \rho(\{t\})$ \cite[Theorem 5.5]{GT},
 \begin{align*}
 &\lim_{y\to 0^+}-iy \int_a^b \int_a^b h(x)^*H(x)W(x,z) \begin{pmatrix} m(z) & \frac{1}{2}I \\ \frac{1}{2}I & 0 \end{pmatrix} W(s,\overline z)^*H(s)h(s) \, ds\, dx \\
 &=\int_a^b \int_a^b h(x)^*H(x)W(x,t) \begin{pmatrix} \rho(\{t\}) & 0 \\ 0 & 0 \end{pmatrix} W(s,t)^*H(s)h(s) \, ds\, dx \\
 &= \int_a^b \int_a^b h(x)^*H(x)u(x,t) \rho(\{t\}) u(s,t)^*H(s)h(s) \, ds\, dx \\
 &=((Uh)(t))^* \rho(\{t\}) (Uh)(t) .
 \end{align*}
Combining this with Equation \ref{fc} gives that
 \[
 ((Uh)(t))^* \rho(\{t\}) (Uh)(t)= \begin{cases} \sum_{k=1}^{M(j)} |\langle \varphi_{jk},h \rangle |^2 & t=t_j\in \sigma(S^{\alpha, \beta}) \\ 0 & t \notin \sigma(S^{\alpha, \beta}) \end{cases} .
 \]
Since $\{ \varphi_{jk} \}$ is an orthonormal basis of $\overline{D(\mathcal S^{\alpha, \beta})}=\mathcal S^{\alpha, \beta}(0)^\perp$, this means that $U$ maps $\overline{D(\mathcal S^{\alpha, \beta})}$ isometrically into $L^2(\rho)$ and $\ker U= \mathcal S^{\alpha, \beta}(0)$.

A similar argument gives the formula for $\rho(\{t_j\})$. Let $z=t_j+iy$ and $h\in L_H^2(a,b)$. By the functional calculus,
\[
\lim_{y\to 0^+}-iy(S^{\alpha, \beta}-z)^{-1}h = \sum_{k=1}^{M(j)} \langle \varphi_{jk},h \rangle \varphi_{jk}
\]
The limit here is taken with respect to the norm in $L_H^2(a,b)$. As a consequence, there exists a sequence $y_l \searrow 0$ such that
\[
H(x)\lim_{l \to \infty}-iy_l((S^{\alpha, \beta}-z)^{-1}h)(x) = H(x)\sum_{k=1}^{M(j)} \langle \varphi_{jk},h \rangle \varphi_{jk}(x)
\]
almost everywhere. By Theorem \ref{regres} and the arguments above, 
\begin{align*}
&\lim_{y\to 0^+}-iy((S^{\alpha, \beta}-z)^{-1}h)(x)\\
&= \lim_{y\to 0^+}-iy \int_a^b W(x,z) \begin{pmatrix} m(z) & \frac{1}{2}I \\ \frac{1}{2}I & 0 \end{pmatrix} W(s,\overline z)^*H(s)h(s) \, ds\\
&= \int_a^b W(x,t_j) \begin{pmatrix} \rho(\{t_j \}) & 0 \\ 0 & 0 \end{pmatrix} W(s,t_j)^*H(s)h(s) \, ds\\
&= \int_a^b u(x,t_j) \rho(\{t_j \}) u(s,t_j)^*H(s)h(s) \, ds\\
&=  u(x,t_j) \rho(\{t_j \}) (Uh)(t_j) .
\end{align*}
Hence,
\[
H(x)\sum_{k=1}^{M(j)} \langle \varphi_{jk},h \rangle \varphi_{jk}(x) =  H(x)u(x,t_j) \rho(\{t_j \}) (Uh)(t_j)
\]
for almost every $x$. Now, $\sum_{k=1}^{M(j)} \langle \varphi_{jk},h \rangle \varphi_{jk}$ and $u(x,t_j) \rho(\{t_j \}) (Uh)(t_j)$ are solutions of $Ju'=-t_jHu$. So,
\begin{align*}
&J\left( \sum_{k=1}^{M(j)} \langle \varphi_{jk},h \rangle \varphi_{jk}(x) - u(x,t_j) \rho(\{t_j \}) (Uh)(t_j) \right)'\\
&= -t_jH(x)\left( \sum_{k=1}^{M(j)} \langle \varphi_{jk},h \rangle \varphi_{jk}(x) - u(x,t_j) \rho(\{t_j \}) (Uh)(t_j) \right)=0 .
\end{align*}
Thus, $\sum_{k=1}^{M(j)} \langle \varphi_{jk},h \rangle \varphi_{jk}(x) - u(x,t_j) \rho(\{t_j \}) (Uh)(t_j)$ is some constant $c$, but then $H(x)c=0$ almost everywhere. Since $H$ is definite on $(a,b)$, this means that $c=0$, i.e.
\[
\sum_{k=1}^{M(j)} \langle \varphi_{jk},h \rangle \varphi_{jk}(x) = u(x,t_j) \rho(\{t_j \}) (Uh)(t_j)
\] 
for every $x \in (a,b)$. So, since $\varphi_{jk}(x)=u(x,t_j)c_{jk}$, 
\begin{align*}
0&= \sum_{k=1}^{M(j)} \langle \varphi_{jk},h \rangle \varphi_{jk}(x) - u(x,t_j) \rho(\{t_j \}) (Uh)(t_j)\\
&= u(x,t_j) \int_a^b \left( \sum_{k=1}^{M(j)} c_{jk}c_{jk}^*-\rho(\{t_j\}) \right) u(s,t_j)^*H(s)h(s) \, ds 
\end{align*}
for every $x \in (a,b)$. Now, for fixed $x$, $u(x,t_j): \C^n \to \C^{2n}$ is injective, so
\[
\int_a^b \left( \sum_{k=1}^{M(j)} c_{jk}c_{jk}^*-\rho(\{t_j\}) \right) u(s,t_j)^*H(s)h(s) \, ds =0
\]
for every $h \in L_H^2(a,b)$. Hence,
\[
H(x)u(x,t_j)\left( \sum_{k=1}^{M(j)} c_{jk}c_{jk}^*-\rho(\{t_j\}) \right)^*=0
\]
for almost every $x$. So, since $u(x,t_j)\left( \sum_{k=1}^{M(j)} c_{jk}c_{jk}^*-\rho(\{t_j\}) \right)^*$ is a solution of $Ju'=-t_jHu$, $u(x,t_j)\left( \sum_{k=1}^{M(j)} c_{jk}c_{jk}^*-\rho(\{t_j\}) \right)^*$ is a constant, and thus $H(x)$ times that constant equals $0$ almost everywhere. Since $H$ is definite on $(a,b)$, this implies that $u(x,t_j)\left( \sum_{k=1}^{M(j)} c_{jk}c_{jk}^*-\rho(\{t_j\}) \right)^*$ is equal to $0$ for every $x$. The claimed formula $\rho(\{t_j\})=\sum_{k=1}^{M(j)} c_{jk}c_{jk}^*$ follows.

Notice that $\rho(\{t_j\}) (U\varphi_{lm})(t_j)=0$ for $j\neq l$ since 
\[
c_{jk}^*\int_a^b u(x,t_j)^*H(x)\varphi_{lm}(x) \, dx=\langle \varphi_{jk},\varphi_{lm} \rangle =0 .
\]
Also, note that $\rho(\{t_j\})$ has rank $M(j)$ since $\{c_{jk}\}$ is linearly independent. So, $\{ U\varphi_{jk} \}$ is an orthonormal basis for $L^2(\rho)$. Hence, $U$ is surjective.

Finally, to prove that $U\mathcal S^{\alpha, \beta}U^*=M_t$, it suffices to show that $M_tUh=Ug$ for every $(h,g) \in \mathcal S^{\alpha, \beta}$ with $g \in \overline{D(\mathcal S^{\alpha, \beta})}$, and that $D(\mathcal S^{\alpha, \beta}) = \{ h \in \overline{D(\mathcal S^{\alpha, \beta})}: Uh \in D(M_t) \}$. Fix $t=t_j \in \sigma(\mathcal S^{\alpha, \beta})$ and $(h,g) \in \mathcal S^{\alpha, \beta}$ with $g \in \overline{D(\mathcal S^{\alpha, \beta})}$, and let $h_0$ denotes the unique absolutely continuous representative of $h$ that satisfies $Jh_0'=-Hg$. Then 
\begin{align*}
t(Uh)(t)&=t\int_a^b u(x,t)^*H(x)h(x) \, dx\\
&=\int_a^b u(x,t)'^*Jh_0(x) \, dx\\
&=u(b,t)^*Jh_0(b)-u(a,t)^*Jh_0(a)-\int_a^b u(x,t)^*Jh_0'(x) \, dx\\ 
&=u(b,t)^*Jh_0(b)-u(a,t)^*Jh_0(a)+\int_a^b u(x,t)^*H(x)g(x) \, dx .
\end{align*}
Since $u$ and $h_0$ satisfy the boundary condition at $a$, $u(a,t)^*Jh_0(a)=0$. Also,
\begin{align*}
\rho(\{t\})u(b,t)^*Jh_0(b)&=\sum_{k=1}^{M(j)} c_{jk}c_{jk}^*u(b,t_j)^*Jh_0(b)\\
&=\sum_{k=1}^{M(j)} c_{jk}\varphi_{jk}(b)^*Jh_0(b)=0
\end{align*}
since $\varphi_{jk}$ and $h_0$ satisfy the boundary condition at $b$. Hence, $t(Uh)(t)=(Ug)(t)$ in $L^2(\rho)$. The claim about the domains is obvious since
\[
D(\mathcal S^{\alpha, \beta})=\bigg\{ \sum_j\sum_{k=1}^{M(j)} a_{jk} \varphi_{jk} : \sum_j\sum_{k=1}^{M(j)}|a_{jk}|^2(1+ t_j^2) <\infty \bigg\} ,
\]
\[
\overline{D(\mathcal S^{\alpha, \beta})}=\bigg\{ \sum_j\sum_{k=1}^{M(j)} a_{jk} \varphi_{jk} : \sum_j\sum_{k=1}^{M(j)}|a_{jk}|^2<\infty \bigg\},
\]
and for $h=\sum_j\sum_{k=1}^{M(j)} a_{jk} \varphi_{jk} \in \overline{D(\mathcal S^{\alpha, \beta})}$, with $\sum_j\sum_{k=1}^{M(j)}|a_{jk}|^2$, $Uh=\sum_j\sum_{k=1}^{M(j)} a_{jk} U\varphi_{jk} \in D(M_t)$ if and only if $\sum_j M(j)t_j^2 <\infty$.
\end{proof}

\section{Half line canonical systems}
In this section, assume that $H$ is definite on $(a,b)=(0,\infty)$, with $0$ a regular endpoint, and that both deficiency indices are equal to $n$. An equivalent condition, which is used below, to having deficiency indices both equal to $n$ is provided by the following theorem. The theorem is well-known in the order two case \cite{BHS, Rembook}, but it does not appear in the literature for general $n$. Recall that for $(f,g)$ in the maximal relation $\mathcal T$, $f_0$ denotes the unique absolutely continuous representative of $f$ that satisfies $Jf_0'=-Hg$, as introduced in \eqref{maxrel}. This notation is used throughout this section.
\begin{Theorem}
\label{lpcriterion}
Let $H$ be definite on $(0,\infty)$, with $0$ a regular endpoint. Then $H$ has deficiency indices both equal to $n$ if and only if 
\[
\lim_{x\to \infty} f_0(x)^*Jh_0(x)=0
\]
for all $(f,g), (h,k) \in \mathcal T$.
\end{Theorem}
\begin{proof}
Suppose $H$ has both deficiency indices equal to $n$. Let $N=\begin{pmatrix} -I & 0 \\ 0 & I \end{pmatrix}$ and $H_N(x)=NH(-x)N$ for $x\in (-\infty,0)$. A quick calculation shows that $H_N\ge 0$ and that $u$ solves $Ju'=-Hv$ if and only if $y(x)=Nu(-x)$ solves $Jy'=-H_Nw$ where $w(x)=Nv(-x)$. Let
\[
\tilde H(x) =\begin{cases} H_N(x) & x < 0\\ H(x) & x> 0 \end{cases} .
\]
Consider the whole line canonical system $\tilde H$ on $(a,b)=(-\infty,\infty)$, whose maximal relation $\mathcal T_{\tilde H}$ is defined, as always, by \eqref{maxrel}. So, by the above observation about solutions, $(f,g),(h,k) \in \mathcal T_H$ extend in the obvious way to $(\tilde f,\tilde g),(\tilde h,\tilde k) \in \mathcal T_{\tilde H}$. By \cite[Proposition 5.4]{Lesch}, $\tilde H$ has both deficiency indices equal to $0$. Thus, $\mathcal T_{\tilde H} = {\mathcal T_{\tilde H}}^*$. So, $(\tilde f,\tilde g),(\tilde h,\tilde k)$ are in the minimal relation ${\mathcal T_{\tilde H}}^*$. It follows that $\lim_{x\to \infty} \tilde f_0(x)^*J\tilde h_0(x)=0$ by \cite[Proposition 4.10]{BHSW}. Hence, $\lim_{x\to \infty} f_0(x)^*Jh_0(x) =0$. 

Conversely, suppose that $\lim_{x\to \infty} f_0(x)^*Jh_0(x)=0$ for all $(f,g), (h,k) \in \mathcal T_H$. Take any definite canonical system $G$ on $(-\infty, 0)$ such that $0$ is a regular endpoint and both deficiency indices are equal to $n$ (one could take $G(x)=I$). Let
\[
\tilde H(x) =\begin{cases} G(x) & x < 0\\ H(x) & x> 0 \end{cases} .
\]
Again, consider the whole line canonical system $\tilde H$ with its maximal relation $\mathcal T_{\tilde H}$ defined by \eqref{maxrel}. Suppose $(\tilde f,\tilde g),(\tilde h,\tilde k) \in \mathcal T_{\tilde H}$. Then their obvious restrictions are in $\mathcal T_G$ and $\mathcal T_H$. So, by the first step, $\lim_{x\to -\infty} \tilde f_0(x)^*J\tilde h_0(x)=0$. By assumption, $\lim_{x\to \infty} \tilde f_0(x)^*J\tilde h_0(x)=0$. So, by \cite[Proposition 4.10]{BHSW}, $(\tilde f,\tilde g)$ is in the minimal relation ${\mathcal T_{\tilde H}}^*$. Hence, $\mathcal T_{\tilde H} \subset {\mathcal T_{\tilde H}}^*$. So, since the minimal relation ${\mathcal T_{\tilde H}}^*$ is also a subset of the maximal relation $\mathcal T_{\tilde H}$, $\mathcal T_{\tilde H} = {\mathcal T_{\tilde H}}^*$. Hence, there are $n$ linearly independent $L_{\tilde H}^2(0,\infty)=L_H^2(0,\infty)$ solutions of $Ju'=-iHu$, and likewise for $Ju'=iHu$, by \cite[Equation 3.15, Equation 3.21, Corollary 4.20]{BHSW}. So, since $H$ is definite on $(0,\infty)$, $H$ has deficiency indices both equal to $n$ by \cite[Proposition 2.19]{Lesch}.   
\end{proof}  

In the rest of this section, fix a boundary condition $\alpha \in \C^{n\times 2n}$ at $0$ with the properties given in \eqref{bc}. The self-adjoint relation, defined as in \eqref{lpcs}, generated by this boundary condition $\alpha$ is denoted $\mathcal S$. As in the previous section, $u$ and $v$ denote the solutions of $\eqref{can}$ with 
\[
u(0,z)=-J\alpha^* , \quad v(0,z)=\alpha^*.
\]
\begin{Theorem}
\label{lpm}
For every $z \in \C/\R$, there exists a unique matrix $m(z)\in  \C^{n\times n}$ such that the columns of $f_m(x,z)=v(x,z)+u(x,z)m(z)$ are in $L_H^2(0,\infty)$. Moreover, $\Im(m(z))>0$ for $z\in \C^+$ and $m(\overline z)=m(z)^*$. 
\end{Theorem}
\begin{proof}
Let $z \in \C/\R$. Since $H$ is definite on $(0,\infty)$ and the deficiency indices are $n$, there exist $n$ linearly independent solutions of \eqref{can} in $L_H^2(0,\infty)$ by \cite[Proposition 2.19]{Lesch}. Hence, there exist $a(z),b(z)\in \C^{n\times n}$ such that $u(x,z)a(z)+v(x,z)b(z)$ has rank $n$ and all of its columns in $L_H^2(0,\infty)$. I claim $b(z)$ is invertible. Suppose $c \in \ker(b(z))$. Then $u(x,z)a(z)c$ is in $L_H^2(0,\infty)$, solves \eqref{can}, and satisfies the boundary condition $\alpha$ at $0$. Since $z\in \C/\R$ cannot be an eigenvalue of the self-adjoint relation $\mathcal S$, this implies that $u(x,z)a(z)c$ represents the zero element of $L_H^2(0,\infty)$. Since $H$ is definite on $(0,\infty)$ and $u(x,z)a(z)c$ is a solution of \eqref{can}, it follows that $a(z)c=0$. Thus, $c \in \ker(u(x,z)a(z)+v(x,z)b(z))$. So, since $u(x,z)a(z)+v(x,z)b(z) \in \C^{2n\times n}$ has rank $n$, $c=0$. Hence, $b(z)$ is invertible and, by defining $m(z)=a(z)b(z)^{-1}$, one obtains that $f_m(x,z)=v(x,z)+u(x,z)m(z)$ has all of its columns $L_H^2(0,\infty)$.

To show that $m(z)$ is unique, fix $z \in \C/\R$, and suppose that $\tilde m\in  \C^{n\times n}$ is such that the columns of $v(x,z)+u(x,z)\tilde m$ are in $L_H^2(0,\infty)$. Then the columns of $u(x,z)(m(z)-\tilde m)$ are in $L_H^2(0,\infty)$, and they satisfy the boundary condition $\alpha$ at $0$. Since $z \in \C/\R$ cannot be an eigenvalue of the self-adjoint relation $\mathcal S$, and $H$ is definite on $(0,\infty)$, this implies that $m(z)-\tilde m=0$.

We now show that $\Im(m(z))>0$ for $z\in \C^+$. Since $f_m$ is a solution of \eqref{can}, $(f_m^*Jf_m)'=-2i\Im(z)f_m^*Hf_m$. Thus, 
{%
\small
\[
f_m(b,z)^*Jf_m(b,z)-f_m(0,z)^*Jf_m(0,z)=-2i\Im(z) \int_0^b f_m(x,z)^*H(x)f_m(x,z) \, dx 
\]
}%
for all $b\in (0,\infty)$. Note that each column $f_j$ of $f_m$ is absolutely continuous and solves \eqref{can}. So, $(f_j, zf_j) \in \mathcal T$, the maximal relation defined by \eqref{maxrel}, and ${f_j}_0=f_j$. So, by Theorem \ref{lpcriterion}, $\lim_{b\to \infty} f_m(b,z)^*Jf_m(b,z)=0$. A calculation shows that $f_m(0,z)^*Jf_m(0,z)=m(z)-m(z)^*$. So, for $z\in \C^+$, 
\[
\Im(m(z))=\Im(z) \int_0^\infty f_m(x,z)^*H(x)f_m(x,z) \, dx>0
\]
since $H$ is definite on $(0,\infty)$.

To prove the last equation in the theorem, note that, as in the proof of Theorem \ref{regm},  $f_m(x,z)^*Jf_m(x,\overline z)$ is a constant function of $x$. By the same argument as above, $\lim_{b\to \infty} f_m(b,z)^*Jf_m(b,\overline z)=0$. Hence, 
\[
0=f_m(0,z)^*Jf_m(0,\overline z)=m(\overline z)-m(z)^* .
\]
\end{proof}

\begin{Theorem}
\label{lpres}
Define
\[
G(x,y,z) = \begin{cases} u(x,z)f_m(y,\overline z)^* & x\leq y \\ f_m(x,z)u(y,\overline z)^* & x>y \end{cases}
\]
for $z \in \C / \R$. Then 
\[
((\mathcal S -z)^{-1}h)(x)= \int_0^\infty G(x,y,z)H(y)h(y) \, dy
\]
for all $h\in L_H^2(0,\infty)$. Moreover, $g(x)=\int_0^\infty G(x,y,z)H(y)h(y) \, dy$ is the unique absolutely continuous representative of $(\mathcal S - z)^{-1}h$ that solves $Jg'=-zHg-Hh$.
\end{Theorem}
\begin{remark}
The last claim in the theorem is also true in the situation of Theorem \ref{regres} and is essentially trivial to prove, but it is included here for convenience since it is used later in this section.
\end{remark}
\begin{proof}
The same argument as in the proof of Theorem \ref{regres} shows that $g$ is absolutely continuous, that is solves $Jg'=-zHg-Hh$, and that it satisfies the boundary condition $\alpha$ at $0$. Moreover, if $h$ has compact support, say in $[0,L]$, then $g(x)=f_m(x,z)\int_0^L u(y,\overline z)^*H(y)h(y) \, dy$ for $x\ge L$. Hence, for $h$ with compact support, $g \in L_H^2(0,\infty)$, and thus $(g,h) \in \mathcal S-z$. So, $(\mathcal S-z)^{-1}h=g$ for $h$ with compact support.

Now, fix any $h\in L_H^2(0,\infty)$, and take $h_j$ with compact support such that $h_j \to h$ in $L_H^2(0,\infty)$. For fixed $x$, the rows of $G(x, \cdot, z)$ are adjoints of elements of $L_H^2(0,\infty)$. Hence,  
\[
\lim_{j\to \infty} \int_0^\infty G(x,y,z)H(y)h_j(y) \, dy = \int_0^\infty G(x,y,z)H(y)h(y) \, dy
\]
pointwise. Since $\lim_{j\to \infty} (\mathcal S-z)^{-1}h_j=(\mathcal S-z)^{-1}h$ in $L_H^2(0,\infty)$, there exists a subsequence such that
\[
H(x)\lim_{k\to \infty} ((\mathcal S-z)^{-1}h_{j_k})(x)= H(x)((\mathcal S-z)^{-1}h)(x)
\]
almost everywhere. Thus, for almost every $x\in (0,\infty)$,
\[
H(x)((\mathcal S-z)^{-1}h)(x)=H(x)\int_0^\infty G(x,y,z)H(y)h(y) \, dy=H(x)g(x) .
\]
So, $(\mathcal S-z)^{-1}h=g$ as elements of $L_H^2(0,\infty)$.

Finally, to verify the last claim in the theorem, suppose that $k(x)$ is absolutely continuous, $H(x)k(x)=H(x)g(x)$ almost everywhere, and $Jk'=-zHk-Hh$. Then $u=k-g$ solves $Ju'=-zHu$ and $H(x)u(x)=0$ almost everywhere. Since $H$ is definite on $(0,\infty)$, this implies that $u(x)=0$ for all $x>0$. Thus, $k(x)=g(x)$ for all $x>0$.
\end{proof}

\begin{Theorem}
\label{mhol}
The function $m(z)$ from Theorem \ref{lpm} is holomorphic on $\C/\R$.
\end{Theorem}
\begin{remark}
Note that $u(x,z)$ and $v(x,z)$ are holomorphic by the basic theory of ordinary differential equations since they solve \eqref{can} and their initial values are constant functions of $z$. The function $f_m(x,z)$ is also a solution of \eqref{can}, but its initial value is $v(0,z)+u(0,z)m(z)$, and it is not yet known, before doing the proof below, whether $m(z)$ is holomorphic. So, one cannot use the same basic theory of ordinary differential equations here to say that $f_m(x,z)$ is holomorphic and, hence, $m(z)$ is holomorphic; that would be circular reasoning. It also does not follow from ordinary differential equation theory that $f_m(x,z)$ is holomorphic because it is in $L_H^2(0,\infty)$: one could multiply $f_m(x,z)$ by any function of $z$ to obtain another $L_H^2(0,\infty)$ solution, which certainly does not have to be holomorphic.
\end{remark}
\begin{proof}
The structure of the proof is as follows. First, it is shown that $g(0,z)=\int_0^\infty G(0,y,z)H(y)h(y) \, dy$ is holomorphic for any $h\in L_H^2(0,\infty)$ by using the fact that $\langle k, (\mathcal S - z)^{-1}h \rangle$ is holomorphic for any $k\in L_H^2(0,\infty)$ and then choosing a suitable $k$. It is then shown that, since $u(x,z)$, $v(x,z)$, and $g(0,z) = u(0,z) \int_0^\infty \left( v(x,\overline z) + u(x,\overline z)m(\overline z) \right)^*H(x)h(x) \, dx$, for every $h\in L_H^2(0,\infty)$, are holomorphic, $m(z)$ must also be holomorphic.  

Note that $\langle k, (\mathcal S - z)^{-1}h \rangle$ is holomorphic on $\C/\R$ for any $h,k\in L_H^2(0,\infty)$. This is because it holds for $h,k \in \overline{D(\mathcal S)}$ by the functional calculus and Morera's theorem, and if either $h$ or $k$ is in $\mathcal S(0)$, then $\langle h, (\mathcal S - z)^{-1}k \rangle=0$. Fix any $h\in L_H^2(0,\infty)$ and $(k,l)\in \mathcal T$, and let $g(x,z)=\int_0^\infty G(x,y,z)H(y)h(y) \, dy$. So, by Theorem \ref{lpres}, 
\begin{align*}
-z\langle k, (\mathcal S - z)^{-1}h \rangle &= -z\int_0^\infty k_0(x)^*H(x)g(x,z) \, dx\\
&=\int_0^\infty k_0(x)^*\left( Jg'(x,z)+H(x)h(x) \right) \, dx .
\end{align*}
So, $\int_0^\infty k_0(x)^*\left( Jg'(x,z)+H(x)h(x) \right) \, dx$ is holomorphic. Thus, since $k_0$ and $h$ are in $L_H^2(0,\infty)$ and do not depend on $z$, $\int_0^\infty k_0(x)^*Jg'(x,z) \, dx$ is holomorphic. By an integration by parts and Theorem \ref{lpcriterion},
\begin{align*}
\int_0^\infty k_0(x)^*Jg'(x,z) \, dx &=-k_0(0)^*Jg(0,z) - \int_0^\infty k_0(x)'^*Jg(x,z) \, dx \\
&=-k_0(0)^*Jg(0,z) - \int_0^\infty l(x)^*H(x)g(x,z) \, dx
\end{align*}
So, since $\langle l, (\mathcal S - z)^{-1}h \rangle = \int_0^\infty l(x)^*H(x)g(x,z) \, dx$ is holomorphic, the boundary term $-k_0(0)^*Jg(0,z)$ is holomorphic, and this holds for any $(k,l)\in \mathcal T$. For any $c\in \C^{2n}$, there exists $(k,l)\in \mathcal T$ with $k_0(0)=c$ since $H$ is definite on $(0,\infty)$ and $0$ is a regular endpoint \cite[Corollary 2.16]{BHSW}. Hence,  $g(0,z)$ is holomorphic.

Now, 
\begin{align*}
g(0,z) &= u(0,z) \int_0^\infty f_m(x,\overline z)^*H(x)h(x) \, dx\\
&= -J\alpha^* \int_0^\infty \left( v(x,\overline z) + u(x,\overline z)m(\overline z) \right)^*H(x)h(x) \, dx .  
\end{align*}
So, $ \int_0^\infty \left( v(x,\overline z) + u(x,\overline z)m(\overline z) \right)^*H(x)h(x) \, dx$ is holomorphic. Let $(0,b)$ be a bounded interval on which $H(x)$ is definite. Let $h \in L_H^2(0,b)$. Then, since $(0,b)$ is bounded, 
\[
\int_0^b v(x,\overline z)^*H(x)h(x) \, dx + m(\overline z)^*\int_0^b u(x,\overline z)^*H(x)h(x) \, dx
\]
is holomorphic. Since $v(x,\overline z)^*$ is holomorphic and $(0,b)$ is bounded, the integral $\int_0^b v(x,\overline z)^*H(x)h(x) \, dx$ is holomorphic by Morera's theorem. Thus, since $m(\overline z)^*=m(z)$ by Theorem \ref{lpm}, $m(z)\int_0^b u(x,\overline z)^*H(x)h(x) \, dx$ is holomorphic. This is true for any $h \in L_H^2(0,b)$. So, to show that $m(z)$ is holomorphic, it suffices to prove that the map 
\[
h \mapsto \int_0^b u(x,\overline z)^*H(x)h(x) \, dx , \quad L_H^2(0,b) \to \C^n ,
\] 
is surjective. Suppose it is not. Take $c\in \C^n$, $c\neq 0$, orthogonal to the image. Then $c^*\int_0^b u(x,\overline z)^*H(x)h(x) \, dx =0$ for all $h\in L_H^2(0,b)$. Thus, $H(x)u(x,\overline z)c=0$ for almost every $x \in (0,b)$. So, $J(u(x,\overline z)c)'=0$ almost everywhere. Hence, $u(x,\overline z)c$ is some constant $k$. So, since $H(x)k=0$ for almost all $x\in (0,b)$ and $H$ is definite there, $u(x,\overline z)c=k=0$. But the matrix $u(x,\overline z)$ has a left inverse, so $c=0$, a contradiction.  
\end{proof}

The following lemma asserts that, for the self-adjoint relation $\mathcal S$ defined by \eqref{lpcs}, every spectral representation $U: L_H^2(0,\infty) \to \bigoplus_j L^2(\mu_j)$ is a kind of generalized eigenfunction expansion. Note that $\bigoplus_j L^2(\mu_j)$ could be a finite or infinite sum. If the spectral representation is ordered, meaning $\mu_{j+1}\ll \mu_j$ for all $j$, and $\mu_j$ is not the $0$ measure for all $j$, then the lemma asserts that there are at most $n$ measures $\mu_j$. Recall that $H(x)\in \C^{2n\times 2n}$, and the deficiency indices are assumed in this section to equal $n$. An ordered spectral representation always exists \cite[Theorem 8.1]{WMLN}.
\begin{Lemma}
\label{replemma1}
Suppose $U: L_H^2(0,\infty) \to \bigoplus_j L^2(\mu_j)$ is a spectral representation of $\mathcal S$ with each $\mu_j$ a Borel measure on $\R$. Then there exist measurable functions $u_j(x,t)$ such that, for $\mu_j$-almost every $t$, $u_j(x,t)$ is an absolutely continuous function of $x\in (0,\infty)$, $Ju_j'=-tHu_j$, $\alpha u_j(0,t)=0$, and
\[
 (U_jh)(t)=\int_0^\infty u_j(x,t)^*H(x)h(x) \, dx
 \] 
for $h\in L_H^2(0,\infty)$ that have a representative with compact support. If $\mu_{j+1}\ll \mu_j$ and $\mu_j$ is not the $0$ measure for all $j$, then there are $k \le n$ measures $\mu_j$, $1\le j \le k$, and, for all $l \le k$, the solutions $u_j(x,t)$, $1\le j \le l$, are linearly independent for $\mu_l$-almost every $t$.
\end{Lemma}
\begin{proof}
A summary of the first and most technical part of the proof, obtaining the integral kernel $u_j(x,t)$ for $U_j$, is as follows. By Theorem \ref{lpres}, the resolvent of $\mathcal S$ restricted to any bounded interval is Hilbert-Schmidt. Hence, the restriction of $M_{(t-z)^{-1}}U_j$, where $M_{(t-z)^{-1}}$ is the operator of multiplication by $\frac{1}{t-z}$ in $L^2(\mu_j)$, to any bounded interval is Hilbert-Schmidt, and so it has an associated integral kernel. We will then glue these integral kernels, multiplied by $t-z$, together to obtain an integral kernel for $U_j$. By carefully selecting representatives of this integral kernel when restricted to bounded intervals, we will obtain the solutions $u_j(x,t)$. 

First, we obtain an integral kernel for the restriction of $M_{(t-z)^{-1}}U_j$ to any bounded interval. Let $[a,b] \subset [0,\infty)$ and $z\in \C/\R$. Define $(R^{a,b}h)(x)= ((\mathcal S - z)^{-1}\chi_{(a,b)}h)(x)$ and $(U_j^{a,b}h)(x)= (U_j\chi_{(a,b)}h)(x)$ for $h\in L_H^2(0,\infty)$. By Theorem \ref{lpres}, $R^{a,b}$ is Hilbert-Schmidt. So, since $M_{(t-z)^{-1}}U_j^{a,b}=U_jR^{a,b}$ and $U_j$ is bounded,  
$M_{(t-z)^{-1}}U_j^{a,b}$ is Hilbert-Schmidt. Hence, there exists a measurable function $K_j^{a,b}(x,t)$ such that 
\begin{equation}
\label{HS}
\int_{-\infty}^\infty \int_0^\infty K_j^{a,b}(x,t)^*H(x)K_j^{a,b}(x,t) \, dx \, d\mu_j(t) < \infty
\end{equation}
and
\[
\frac{1}{t-z}(U_j^{a,b}h)(t)=\int_0^\infty K_j^{a,b}(x,t)^*H(x)h(x) \, dx
\]
for all $h\in L_H^2(0,\infty)$ and $\mu_j$-almost every $t$. Note that since $U_j^{a,b}\chi_{(a,b)}h=U_j^{a,b}h$, 
\[
\int_0^\infty K_j^{a,b}(x,t)^*H(x)h(x) \, dx = \int_a^b K_j^{a,b}(x,t)^*H(x)h(x) \, dx 
\]
for all $h\in L_H^2(0,\infty)$ and $\mu_j$-almost every $t$. Suppose $[a,b]\subset [c,d] \subset [0,\infty)$. Then, since $U_j^{a,b}=U_j^{c,d}\chi_{(a,b)}$,
\[
\int_a^b K_j^{a,b}(x,t)^*H(x)h(x) \, dx = \int_a^b K_j^{c,d}(x,t)^*H(x)h(x) \, dx 
\]
for all $h\in L_H^2(0,\infty)$ and $\mu_j$-almost every $t$. 

We now glue these integral kernels, multiplied by $t-z$, together. Define $\tilde u_j(x,t)=(t-z)K_j^{l-1,l}(x,t)$ for $x \in [l-1,l)$, $l \in \N$. Suppose $h\in L_H^2(0,\infty)$ is supported in $[0,l)$ for some $l\in \N$. Then
\begin{align*}
(U_jh)(t)= (U_j^{0,l}h)(t) &= \int_0^\infty (t-z)K_j^{0,l}(x,t)^*H(x)h(x) \, dx\\
 &= \sum_{i=1}^l \int_{i-1}^i (t-z)K_j^{0,l}(x,t)^*H(x)h(x) \, dx\\
 &= \sum_{i=1}^l \int_{i-1}^i (t-z)K_j^{i-1,i}(x,t)^*H(x)h(x) \, dx\\
 &= \sum_{i=1}^l \int_{i-1}^i \tilde u_j(x,t)^*H(x)h(x) \, dx\\
 &= \int_0^\infty \tilde u_j(x,t)^*H(x)h(x) \, dx .
\end{align*}
for $\mu_j$-almost every $t$. Hence, $(U_jh)(t)=\int_0^\infty \tilde u_j(x,t)^*H(x)h(x) \, dx$ $\mu_j$-almost everywhere for all $h\in L_H^2(0,\infty)$ with compact support.

Next, we show that $\tilde u_j(\cdot, t)\big |_{(0,l)}$ is in the maximal relation for the canonical system on any sufficiently large bounded interval $(0,l)$, and then we will choose absolutely continuous representatives of $\tilde u_j(\cdot, t)\big |_{(0,l)}$ that solve the canonical system on their respective intervals. Since $(U_jh)(t) < \infty$ for $\mu_j$-almost every $t$, $\tilde u_j(x,t) \in L_H^2(0,b)$ for every finite $b$ and $\mu_j$-almost every $t$. Take $l\in \N$ such that $H$ is definite on $(0,l)$. Consider the minimal relation $\mathcal T_0(0,l)$ on $(0,l)$. Let $(\tilde f, \tilde g) \in \mathcal T_0(0,l)$. It follows from \cite[Proposition 4.14]{BHSW} that $\tilde f_0(0)=0=\tilde f_0(l)$. Define
\[
(f(x),g(x))=\begin{cases} (\tilde f_0(x),\tilde g(x)) & x < l\\ (0,0) & x\ge l \end{cases} .
\]
Then $(f,g) \in \mathcal T_0(0,\infty)$, again by \cite[Proposition 4.14]{BHSW}. So, since $U_jg=M_tU_jf$, it follows that
\begin{align*}
\langle t\tilde u_j(\cdot, t)\big |_{(0,l)},\tilde f \rangle_{L_H^2(0,l)} &= t(U_jf)(t) \\
&=(U_jg)(t)\\
&=\int_0^\infty \tilde u_j(x,t)^*H(x)g(x) \, dx\\
&= \langle \tilde u_j(\cdot, t)\big |_{(0,l)},\tilde g \rangle_{L_H^2(0,l)}
\end{align*}
for $\mu_j$-almost every $t$. Thus, $(\tilde u_j(\cdot, t)\big |_{(0,l)}, t\tilde u_j(\cdot, t)\big |_{(0,l)})$ is in the maximal relation $\mathcal T(0,l)$ on $(0,l)$ for $\mu_j$-almost every $t$. Hence, for $\mu_j$-almost every $t$, there exists an absolutely continuous function $u_j^l(x,t)$ for $x \in (0,l)$ that is uniquely determined by the properties that $H(x)u_j^l(x,t)=H(x)\tilde u_j(x, t)$ for almost every $x\in (0,l)$ and that $u_j^l$ solves $Ju'=-tH\tilde u_j$ on $(0,l)$. Suppose $p\ge l$. Then, for $\mu_j$-almost every $t$, $H(x)u_j^p(x,t)=H(x)\tilde u_j(x, t)$ for almost every $x\in (0,l)$, and $u_j^p$ is a solution of $Ju'=-tH\tilde u_j$ on $(0,l)$. Hence, $u_j^p(\cdot, t)\big |_{(0,l)}=u_j^l(\cdot, t)$ for $\mu_j$-almost every $t$. 

Thus, we are now ready to define $u_j(x,t)$. The set of $t\in \R$ such that $u_j^p(\cdot, t)$ does not exist for some integer $p\ge l$ is contained in a measurable set, say $N_j$, with $\mu_j(N_j)=0$. Define $u_j(x,t)=u_j^p(x,t)$ if $x \in (0,p)$ and $t \notin N_j$ and $u_j(x,t)=0$ for $t\in N_j$. Then the above argument shows that $u_j$ is well-defined, measurable, and a solution of $Ju_j'=-tHu_j$ for $\mu_j$-almost every $t$. Also, since $H(x)u_j(x,t)=H(x)\tilde u_j(x,t)$ for almost every $x\in (0,\infty)$ and $\mu_j$-almost every $t$, $(U_jh)(t)=\int_0^\infty u_j(x,t)^*H(x)h(x) \, dx$ $\mu_j$-almost everywhere for all $h\in L_H^2(0,\infty)$ with compact support.

We now check that $u_j(x,t)$ satisfies the boundary condition $\alpha$ at $0$. Let $(f,g) \in \mathcal S$ be such that $f_0(0)=0$ eventually. Then, since also $H(x)g(x)=Jf_0'(x)=0$ eventually, $t\int_0^\infty u_j(x,t)^*H(x)f(x) \, dx = \int_0^\infty u_j(x,t)^*H(x)g(x) \, dx$ for $\mu_j$-almost every $t$. But an integration by parts shows that 
\[
\int_0^\infty u_j(x,t)^*H(x)g(x) \, dx=u_j(0,t)^*Jf_0(0)+t\int_0^\infty u_j(x,t)^*H(x)f(x) \, dx
\]
for $\mu_j$-almost every $t$. Hence, $u_j(0,t)^*Jf_0(0)=0$ for $\mu_j$-almost every $t$. Take $l<\infty$ such that $H$ is definite on $(0,l)$. Using \cite[Corollary 2.16]{BHSW}, take $\tilde g_i \in L_H^2(0,l)$, $i=1,\, \dots,\, n$, and solutions $\tilde f_i$ of $J\tilde f_i'=-H\tilde g_i$ such that $\tilde f_i(0)=J\alpha^*e_i$ and $\tilde f_i(l)=0$. Define
\[
(f_i(x),g_i(x))=\begin{cases} (\tilde f_i(x),\tilde g_i(x)) & x < l\\ (0,0) & x\ge l \end{cases} .
\]
Then $(f_i,g_i) \in \mathcal S$ and $f_i=0$ eventually. Hence, $0=u_j(0,t)^*Jf_i(0)=-u_j(0,t)^*\alpha^*e_i$ for every $i=1,\, \dots,\, n$ and $\mu_j$-almost every $t$. So, $\alpha u_j(0,t)=0$ for $\mu_j$-almost every $t$.

Finally, suppose that, for all $j$, $\mu_{j+1}\ll \mu_j$ and $\mu_j$ is not the $0$ measure. Let $l \le k$ if there are $k<\infty$ measures $\mu_j$, and let $l<\infty$ if there are infinitely many measures in the representation. Suppose that $\sum_{j=1}^l c_j(t)u_j(x,t)=0$. Then, for $\mu_l$-almost every $t$, $\sum_{j=1}^l c_j(t)(U_jh)(t)=0$ for every $h \in L_H^2(0,\infty)$. Using the assumption that $U$ is surjective, take $h \in L_H^2(0,\infty)$ such that $(U_{j_0}h)(t)=\chi_{(S,T)}(t)$, $-\infty<S<T<\infty$, and $U_jh=0$ for $j \neq j_0$. Hence, $c_{j_0}(t)=0$ for $\mu_l$-almost every $t\in (S,T)$. Thus, $c_j(t)=0$ for $\mu_l$-almost every $t$ and every $1\le j \le l$. So, the solutions $u_j(x,t)$, $1\le j \le l$, are linearly independent $\mu_l$-almost everywhere. Since there are only $n$ linearly independent solutions of \eqref{can} that satisfy the boundary at $0$, this shows that there are at most $n$ measures in the ordered spectral representation.
\end{proof}

In the next lemma, an arbitrary ordered spectral representation, which must be a generalized eigenfunction expansion in the sense of Lemma \ref{replemma1}, is written in terms of the fixed matrix solution $u(x,t)$ of \eqref{can} satisfying $u(0,t)=-J\alpha^*$.  
\begin{Lemma}
\label{replemma2}
Let $u(x,t)$ be as in Theorem \ref{lpm}. Take a spectral representation $\tilde U$ with ordered $\mu_,\, \dots,\,\mu_k$ and $u_1,\, \dots,\, u_k$ as in Lemma \ref{replemma1}. Then there exists a measurable function $c:\R \to \C^{n\times k}$ such that, for every column $c_j$ of $c$ and $\mu_j$-almost every $t$, $u_j(x,t)=u(x,t)c_j(t)$ for every $x\in (0,\infty)$. For every compact $[S,T] \subset \R$, $\int_S^T |c_j(t)|^2 \, d\mu_j(t)< \infty$. Define $d\rho(t)=c(t)\begin{pmatrix}
    d\mu_1(t) & & \\
    & \ddots & \\
    & & d\mu_k(t)
  \end{pmatrix}c(t)^*$ and $(Uh)(t)=\int_0^\infty u(x,t)^*H(x)h(x) \, dx$ for $h\in L_H^2(0,\infty)$ with compact support. Then $U$ extends to a map $L_H^2(0,\infty) \to L^2(\rho)$ that provides a spectral representation of $\mathcal S$. Furthermore, $U_1=U\big |_{\overline{D(\mathcal S)}}$ has an inverse given by $(U_1^{-1}F)(x)=\int_{-\infty}^\infty u(x,t) d\rho(t)F(t)$ for $F\in L^2(\rho)$ with compact support.
\end{Lemma}
\begin{proof}
First, we define $c(t)$. Let $N_j$ be a $\mu_j$-measure zero set off of which $u_j(x,t)$ is a solution of $Ju_j'=-tHu_j$ that satisfies the boundary condition at $0$. Define $c_j(t)$ by $u_j(0,t)=u(0,t)c_j(t)$ for $t\notin N_j$, and set $c_j(t)=0$ for $t \in N_j$. Since $u_j$ is measurable and $u(0,t)$ is a constant, $c_j(t)$ is measurable.

Next, we check that $\int_S^T |c_j(t)|^2 \, d\mu_j(t)< \infty$. By \eqref{HS},
\[
\int_{-\infty}^\infty \int_a^b \frac{1}{|t-z|^2} u_j(x,t)^*H(x)u_j(x,t) \, dx \, d\mu_j(t) < \infty 
\]
for any $[a,b]\subset (0,\infty)$ and $z \in \C/\R$.  Hence, for any $[S,T] \subset \R$,
\[
\int_S^T \int_a^b u_j(x,t)^*H(x)u_j(x,t) \, dx \, d\mu_j(t) < \infty .
\]
So, since $u_j(x,t)=u(x,t)c_j(t)$ $\mu_j$-almost everywhere,
\[
\int_S^T c_j(t)^* \int_a^b u(x,t)^*H(x)u(x,t) \, dx \, c_j(t) \, d\mu_j(t) < \infty .
\]
Now, $\int_a^b u(x,t)^*H(x)u(x,t) \, dx$ is a continuous function of $t\in [S,T]$, so, since $[a,b]$ can chosen such that $H$ is definite there, this implies that
\[
\int_S^T |c_j(t)|^2 \, d\mu_j(t) < \infty .
\]

Take $h \in L_H^2(0,\infty)$ with compact support. Note that 
\begin{align*}
\langle {\tilde U}h, {\tilde U}h \rangle &= \sum_{j=1}^k \int_{-\infty}^\infty \bigg | \int_0^\infty u_j(x,t)^*H(x)h(x) \, dx \bigg | ^2 \, d\mu_j(t)\\
&=\int_{-\infty}^\infty \left((Uh)(t)\right)^* \, d\rho(t) (Uh)(t) .
\end{align*}
The above calculation shows that if $h\in \overline{D(\mathcal S)}$ has compact support, then $\langle Uh, Uh \rangle =\langle h, h \rangle$, and if $h\in \mathcal S(0)$ has compact support, then $Uh=0$. Since $L^2(\rho)$ is complete, it follows that $U$ extends to an isometry $\overline{D(\mathcal S)} \to L^2(\rho)$. By defining $Uh=0$ for all $h \in \mathcal S(0)$, one then obtains a linear map $U: L_H^2(0,\infty) \to L^2(\rho)$. To show that it provides a spectral representation, it remains to prove that $U$ is surjective and that $U_1SU_1^{-1}=M_t$, where $S$ is the operator part of $\mathcal S$.

Define $V:L^2(\rho) \to \bigoplus_j L^2(\mu_j)$ by $(VF)(t)=c(t)^*F(t)$. By the definition of $\rho$, this is an isometry. By a similar calculation to that used above to show that $\langle Uh, Uh \rangle = \langle {\tilde U}h, {\tilde U}h \rangle$, $VUh=\tilde Uh$ for $h \in L_H^2(0,\infty)$ with compact support. Since $\tilde U$ and $VU$ are bounded, it follows that $VU=\tilde U$. Hence, $V$ is surjective and thus a bijection. So, $U=V^{-1}\tilde U$ is surjective.

Let $\tilde M_t$ be the multiplication by the variable in $\bigoplus_j L^2(\mu_j)$. Trivially, $VM_t=\tilde M_tV$. So, $M_t=V^{-1}\tilde M_tV=V^{-1}\tilde U_1 S \tilde U_1^{-1}V=U_1SU_1^{-1}$, where $\tilde U_1$ is the restriction of the spectral representation from Lemma \ref{replemma1} to $\overline{D(\mathcal S)}$.

Finally, to obtain the formula for the inverse of $U_1$, take any $h \in L_H^2(0,\infty)$ and $F \in L^2(\rho)$, both with compact support. Write $h=g+k$ with $g\in \overline{D(\mathcal S)}$ and $k\in \mathcal S(0)$. Note that $\langle U_1^{-1}F, h \rangle = \langle U_1^{-1}F, g \rangle$ since $U_1^{-1}F \in \overline{D(\mathcal S)}$. So, since $U_1$ is unitary and $Uk=0$,
\begin{align*}
\langle U_1^{-1}F, h \rangle &= \langle U_1^{-1}F, g \rangle \\
&= \langle F, U_1g \rangle +\langle F, Uk \rangle \\
&= \langle F, Uh \rangle \\
&=\int_{-\infty}^\infty F(t)^* \, d\rho(t) \left(\int_0^\infty u(x,t)^*H(x)h(x) \, dx\right) \\
&=\int_0^\infty \left( \int_{-\infty}^\infty u(x,t) \, d\rho(t) F(t) \right)^* H(x)h(x) \, dx 
\end{align*}
by Fubini's theorem, which is permitted since $Hh \in L^1(0,\infty)$, $F$ has compact support, $u$ is continuous, and $\int_S^T |c_j(t)|^2 \, d\mu_j(t) < \infty$ for compact $[S,T] \subset \R$. Since this holds for all $h \in L_H^2(0,\infty)$ with compact support, the claimed formula for $U_1^{-1}$ follows.
\end{proof}

So, Lemma \ref{replemma2} provides a formula for $U$ when applied to functions in $L_H^2(0,\infty)$ with compact support, and we know that $U: L_H^2(0,\infty) \to L^2(\rho)$ provides a spectral representation of $\mathcal S$, but the target space $L^2(\rho)$ is unknown, in the sense that $\rho$ is defined in terms of the measures $\mu_j$ that were delivered by the abstract proof of the spectral theorem. The next goal is to describe $\rho$ in terms of the canonical system, thus completing the identification of the spectral representation $U: L_H^2(0,\infty) \to L^2(\rho)$. As in the previous section, $\rho$ turns out to be the measure associated with $m(z)$ by its Herglotz representation.     
\begin{Lemma}
\label{Uoff}
Let $f_m=v+um$, as in Theorem \ref{lpm}, and take $U$ as in Lemma \ref{replemma2}. Write $f_m=(f_1\: \cdots \: f_n)$. Then $(Uf_j(\cdot, z))(t)=\frac{1}{t-z}e_j$ for all $z\in \C/\R$.
\end{Lemma} 
\begin{proof}
By the functional calculus, $(\mathcal S -z)^{-1}h=U_1^{-1}M_{(t-z)^{-1}}Uh$ for all $h\in L_H^2(0,\infty)$. Fix $h\in \overline{D(\mathcal S)}$ such that that $Uh$ has compact support. So, $\int_0^\infty G(x,y,z)H(y)h(y) \, dy$ and $\int_{-\infty}^\infty u(x,t) \, d\rho(t) \frac{1}{t-z}(Uh)(t)$ represent the same element in $L_H^2(0,\infty)$. By Fubini's theorem, which is applicable since $Uh$ has compact support and the trace of $\rho$ is a locally finite measure by Lemma \ref{replemma2}, $\int_{-\infty}^\infty \frac{1}{t-z}u(x,t) \, d\rho(t) (Uh)(t)$ is absolutely continuous, and 
\begin{align*}
J\frac{d}{dx}\int_{-\infty}^\infty \frac{1}{t-z}u(x,t) \, d\rho(t) (Uh)(t) &= H(x)\int_{-\infty}^\infty \frac{-t}{t-z}u(x,t) \, d\rho(t) (Uh)(t)\\
&= H(x)\int_{-\infty}^\infty \frac{-z}{t-z}u(x,t) \, d\rho(t) (Uh)(t)\\
 & \quad \: -H(x)\int_{-\infty}^\infty u(x,t) \, d\rho(t) (Uh)(t)\\
&= H(x)\int_{-\infty}^\infty \frac{-z}{t-z}u(x,t) \, d\rho(t) (Uh)(t)\\
& \quad \:-H(x)h(x) .
\end{align*}
But $g(x)=\int_0^\infty G(x,y,z)H(y)h(y) \, dy$ is the unique absolutely continuous representative of $(\mathcal S -z)^{-1}h$ that solves $Jg'=-zHg-Hh$ by Theorem \ref{lpres}, so 
\[
\int_0^\infty G(x,y,z)H(y)h(y) \, dy = \int_{-\infty}^\infty \frac{1}{t-z}u(x,t) \, d\rho(t) (Uh)(t)
\]
for all $x \ge 0$. Setting $x=0$ gives that
\[
u(0,t) \int_0^\infty f_m(y,\overline z)^*H(y)h(y) \, dy = \int_{-\infty}^\infty \frac{1}{t-z}u(0,t) \, d\rho(t) (Uh)(t) .
\]
Hence, since $u(0,t)$ has a left inverse,
\[
\langle f_j(\cdot, \overline z), h \rangle = \int_{-\infty}^\infty \frac{1}{t-z}e_j^* \, d\rho(t) (Uh)(t) .
\]
Write $f_j=k_j+l_j$ with $k_j \in \overline{D(\mathcal S)}$ and $l_j \in \mathcal S(0)$. Then
\[
\langle f_j, h \rangle = \langle k_j, h \rangle = \langle Uk_j, Uh \rangle = \langle Uf_j, Uh \rangle .
\]
So, 
\[
\langle Uf_j(\cdot,\overline z), Uh \rangle =\int_{-\infty}^\infty \frac{1}{t-z}e_j^* \, d\rho(t) (Uh)(t)
\]
for all $h\in \overline{D(\mathcal S)}$ such $Uh$ has compact support. Since functions with compact support are dense in $L^2(\rho)$ and $U$ maps $\overline{D(\mathcal S)}$ onto $L^2(\rho)$, the claim follows.
\end{proof}

\begin{Lemma}
\label{Bformula}
Write $f_j=k_j+l_j$ with $k_j \in \overline{D(\mathcal S)}$ and $l_j \in \mathcal S(0)$, and let $l=(l_1\: \cdots \: l_n)$. Then $\int_0^\infty l(x,z)^*H(x)l(x,z) \, dx$ is constant for $z \in \C^+$.
\end{Lemma}
\begin{proof}
Since $\int_0^\infty l(x,z)^*H(x)l(x,z) \, dx$ is a self-adjoint matrix, it suffices to show that $\int_0^\infty l(x,z)^*H(x)l(x,z) \, dx$ is holomorphic for $z\in \C^+$ because then each entry in the matrix and its complex conjugate are holomorphic functions on $\C^+$. Let $h\in L_H^2(0,\infty)$ be arbitrary. From the proof Theorem \ref{mhol}, $\int_0^\infty f_m(x,\overline z)^*H(x)h(x) \, dx$ is holomorphic on $\C/\R$. So, for $h\in \mathcal S(0)$, $\int_0^\infty l(x,\overline z)^*H(x)h(x) \, dx$ is holomorphic. Now, for $h\in \overline{D(\mathcal S)}=\mathcal S(0)$, $\int_0^\infty l(x,\overline z)^*H(x)h(x) \, dx=0$. So, $\int_0^\infty l(x,\overline z)^*H(x)h(x) \, dx$ is holomorphic for arbitrary $h\in L_H^2(0,\infty)$.  

Since $(0,l_j(\cdot,z)) \in \mathcal S$, there exists an absolutely continuous function $p_j(\cdot, z)$ such that $H(x)p_j(x,z)=0$ and $Jp_j'=-Hl_j$ almost everywhere. Let $p=(p_1\: \cdots \: p_n)$. Take any $(h,g) \in \mathcal T$. Then, by an integration by parts and Theorem \ref{lpcriterion},
\begin{align*}
\int_0^\infty l(x,\overline z)^*H(x)h_0(x) \, dx &= \int_0^\infty p(x,\overline z)'^*Jh_0(x) \, dx\\
&= -p(0,\overline z)^*Jh_0(0) - \int_0^\infty p(x,\overline z)^*Jh_0'(x) \, dx\\
&= -p(0,\overline z)^*Jh_0(0) - \int_0^\infty p(x,\overline z)^*H(x)g(x) \, dx\\
&=-p(0,\overline z)^*Jh_0(0)
\end{align*}
since $H(x)p(x,\overline z)=0$ almost everywhere. By the first paragraph, the matrix $\int_0^\infty l(x,\overline z)^*H(x)h_0(x) \, dx$ is holomorphic. So, $p(0,\overline z)^*Jh_0(0)$ is holomorphic on $\C/\R$. Since $H$ is definite on $(0,\infty)$ and $0$ is a regular endpoint, for any $c \in \C^{2n}$, there exists $(h,g) \in \mathcal T$ with $h_0(0)=c$ \cite[Corollary 2.16]{BHSW}. Thus, $p(0,z)$ is holomorphic. 

Now,
{%
\small 
\begin{align*}
\int_0^\infty l(x,z)^*H(x)l(x,z) \, dx &= \int_0^\infty f_m(x,z)^*H(x)l(x,z) \, dx\\
&= - \int_0^\infty f_m(x,z)^*Jp'(x,z) \, dx\\
&= f_m(0,z)^*Jp(0,z) + \int_0^\infty f_m(x,z)'^*Jp(x,z) \, dx\\
&= f_m(0,z)^*Jp(0,z) + \int_0^\infty \overline{z}f_m(x,z)^*H(x)p(x,z) \, dx\\ 
&= f_m(0,z)^*Jp(0,z) .
\end{align*}
}%
Since $(p_j, l_j) \in \mathcal S$, $p_j$ satisfies the boundary condition at $0$, so $p_j(0,z)=J\alpha^*c_j(z)$ for some $c_j(z)\in \C^n$. Let $c(z)=(c_1(z)\: \cdots \: c_n(z))$. So, since $f_m(0,z)=\alpha^*-J\alpha^*m(z)$,
\[
f_m(0,z)^*Jp(0,z)=(\alpha+m(z)^*\alpha J)(-\alpha^*c(z))=-c(z) .
\]
Now, since $p(0,z)$ is holomorphic by the previous paragraph, and $J\alpha^*$ has a left inverse, $c(z)$ is holomorphic. Thus, $\int_0^\infty l(x,z)^*H(x)l(x,z) \, dx=-c(z)$ is holomorphic. 
\end{proof}

\begin{Theorem}
\label{lprep}
Let $l$ be the multivalued component of $f_m=v+um$, as in Lemma \ref{Bformula}. The Herglotz representation of $m(z)$ is
\[
m(z)=A + Bz + \int_{-\infty}^{\infty} \left( \frac{1}{t-z} - \frac{t}{t^2+1} \right)\, d\rho(t) ,
\]
for some $A=A^*\in \C^{n \times n}$, where $B=\int_0^\infty l(x,z)^*H(x)l(x,z) \, dx$ and $\rho$ is the measure from Lemma \ref{replemma2}. In particular, the map $U: L_H^2(0,\infty) \to L^2(\rho)$, from Lemma \ref{replemma2}, defined by 
\[
(Uh)(t)=\int_0^\infty u(x,t)^*H(x)h(x) \, dx
\]
for $h\in L_H^2(0,\infty)$ with compact support provides a spectral representation of $\mathcal S$ in the $L^2$ space over the measure associated with $m(z)$ by its Herglotz representation.
\end{Theorem}
\begin{proof}
Recall from the proof of Theorem \ref{lpm} that 
\[
\Im{m(z)}= \Im(z) \int_0^\infty f_m(x,z)^*H(x)f_m(x,z) \, dx .
\]
As in Lemma \ref{Bformula}, write $f_j=k_j+l_j$ with $k_j \in \overline{D(\mathcal S)}$ and $l_j \in \mathcal S(0)$, and let $k=(k_1\: \cdots \: k_n)$. Then
\begin{align*}
\int_0^\infty f_m(x,z)^*H(x)f_m(x,z) \, dx=&\int_0^\infty k(x,z)^*H(x)k(x,z) \, dx\\
& +\int_0^\infty l(x,z)^*H(x)l(x,z) \, dx .
\end{align*}
By Lemmas \ref{replemma2} and \ref{Uoff},
\[
\int_0^\infty k(x,z)^*H(x)k(x,z) \, dx=\int_{-\infty}^\infty \frac{1}{|t-z|^2} \, d\rho(t) .
\]
By Lemma \ref{Bformula}, $\int_0^\infty l(x,z)^*H(x)l(x,z) \, dx \ge 0$ is a constant, say $B$, for $z \in \C^+$. So, 
\[
\tilde m(z)= Bz + \int_{-\infty}^{\infty} \left( \frac{1}{t-z} - \frac{t}{t^2+1} \right)\, d\rho(t)
\]
is a Herglotz function, and its imaginary part is
\[
\Im(z)\left( B+ \int_{-\infty}^\infty \frac{1}{|t-z|^2} \, d\rho(t)\right) ,
\]
which is also the imaginary part of $m(z)$. So, since $m(z)$ and $\tilde m(z)$ are holomorphic on $\C^+$,
\[
m(z)=A + Bz + \int_{-\infty}^{\infty} \left( \frac{1}{t-z} - \frac{t}{t^2+1} \right)\, d\rho(t) ,
\]
for some constant $A=A^* \in \C^{n \times n}$. 
\end{proof}
\begin{remark}
Suppose that $H(x)\in\R^{2\times 2}$, $\tr H(x)=1$, and the boundary condition $u_2(0)=0$ is imposed. Then $B>0$ if and only if $H(x)=\begin{pmatrix} 0 & 0 \\ 0 & 1 \end{pmatrix}$ on some interval $(0,b)$, and in that case $B$ is the supremum of such $b$. The traditional proof of this relies on Weyl theory and some results in complex analysis due to de Branges \cite{dB}, \cite{Rembook}. A simple proof can be obtained using the formula for $B$ in the above theorem. An interesting question is whether $B$ can be related to $H$ in some direct way for higher order systems. An attempt to study this using the formula for $B$ runs into the problem that the multivalued part $\mathcal S(0)$ is not as well-understood for higher order systems. For a canonical system on a star graph, the multivalued part and $B$ are described in \cite{dSW}.
\end{remark}
\begin{remark}
Theorem \ref{lprep}, apart from the formula for $B$, is well-known when $n=1$ and $H(x)$ is real. The traditional way to prove it in that case is with a Weyl theory of nested disks. This is done, for example, in the book \cite{Rembook}. There are several technical problems in generalizing such a proof to $n>1$. One obvious problem is that the expected analogue of a Weyl disk is no longer a disk. Some steps, however, are taken in this direction in the papers \cite{HS1}, \cite{HS2}, and \cite{Krall}. As mentioned in the Introduction, a stronger definiteness assumption, which implies that the self-adjoint relation $\mathcal S$ has no multivalued part, is adopted in those papers. Another way to get to the theorem in the case $n=1$ starts by introducing Weyl $m$ functions through the abstract theory of boundary triplets \cite{BHS}. In that approach, substantial work must still be done on the canonical system side to get the spectral representation \cite[Section 7.8, Appendix B]{BHS}. Moreover, there are technical problems in generalizing that approach for $n>1$; for example, the proof, in \cite[Appendix B]{BHS}, that the map giving the spectral representation is surjective does not have an obvious generalization for when $n>1$ and $H$ is not real. The approach used in this paper is inspired by that used in \cite{Teschl} and \cite{WMLN} for equations of a different type than canonical systems.
\end{remark}
\begin{remark}
Whole line canonical systems, assuming the deficiency indices are both equal to $0$, are covered by Theorem \ref{lprep}. To see this, transform $H(x)$ on $(-\infty,0)$ so that it is defined on $(0,\infty)$ by writing $\tilde H(x)=\begin{pmatrix} -I & 0 \\ 0 & I \end{pmatrix}H(-x)\begin{pmatrix} -I & 0 \\ 0 & I \end{pmatrix}$ for $x\in (0,\infty)$. One can collect the two order $2n$ canonical systems $\tilde H(x)$ and $H(x)$ on $(0,\infty)$ into a single order $4n$ canonical system on $(0,\infty)$ with a boundary condition at $0$ that encodes that $\tilde u(0)=\begin{pmatrix} -I & 0 \\ 0 & I \end{pmatrix}u(0)$ for solutions of the original whole line problem. A similar trick will be employed in the next section to turn a canonical system on a compact graph into a canonical system on an interval.
\end{remark}

\section{Canonical systems on compact graphs}
Consider a canonical system on a graph. In this section, the graph $G$ is assumed to be connected and consist of finitely many edges and vertices; it is also assumed here that each edge connects exactly two vertices. For convenience, assume that there is at most one edge between any two vertices. Label the edges $E_1,\, \dots,\, E_k$. With each edge $E_i$, associate a differential equation $Ju'(x) = -zH_i(x)u(x)$ with $x\in (-r_i,r_i)$, $H_i(x)\in\R^{2\times 2}$, $H_i \in L^1(-r_i,r_i)$, and $H_i(x)\ge 0$. Note that $J=\begin{pmatrix} 0 & -1 \\ 1 & 0 \end{pmatrix}$ here.

Suppose the graph is directed: each edge has an initial vertex and terminal vertex. Then interface conditions can be introduced as follows. Fix a vertex $v$. Let $E_o(v)= \{ i : v \textrm{ \rm{is the initial vertex of} }E_i \}$ and $E_t(v)= \{ i : v \textrm{ \rm{is the terminal vertex of} }E_i \}$. Suppose that the numbers in $E_o(v)$ are $i_1<\cdots<i_L$ , $i_p=i_p(v), \, L=L(v)$, and the numbers in $E_t(v)$ are $j_1<\cdots<j_M$, $j_p=j_p(v),\, M=M(v)$. Let $\beta^{(0)}=\beta^{(0)}(v)\in \C^{(L+M)\times 2(L+M)}$ be such that $\beta^{(0)} \beta^{(0)*}=I$ and $\beta^{(0)} J \beta^{(0)*}=0$. Write $\beta^{(0)}=(\beta_1^{(0)} \: \beta_2^{(0)})$ with $\beta_i^{(0)}\in \C^{(L+M)\times (L+M)}$. It will be convenient later to have the entries in $\beta_1^{(0)}$ labelled by
\[
\beta_1^{(0)}=\begin{pmatrix} a_{i_1i_1} & a_{i_1i_2} & \cdots & a_{i_1i_L} & a_{i_1j_1} & \cdots & a_{i_1j_M} \\ a_{i_2i_1} & \cdots \\ \vdots \\ a_{i_Li_1} & \cdots \\ a_{j_1i_1} & \cdots \\ \vdots \\ a_{j_Mi_1} & \cdots \end{pmatrix} , 
\]
and, similarly, the entries in $\beta_2^{(0)}$ by
\[
\beta_2^{(0)}=\begin{pmatrix} b_{i_1i_1} & b_{i_1i_2} & \cdots & b_{i_1i_L} & b_{i_1j_1} & \cdots & b_{i_1j_M} \\ b_{i_2i_1} & \cdots \\ \vdots \\ b_{i_Li_1} & \cdots \\ b_{j_1i_1} & \cdots \\ \vdots \\ b_{j_Mi_1} & \cdots \end{pmatrix} .
\]
An interface condition at $v$ is introduced by requiring that 
\[
\beta_1^{(0)} \begin{pmatrix} u_{(i_1)1}(-r_{i_1}) \\ \vdots \\ u_{(i_L)1}(-r_{i_L}) \\ -u_{(j_1)1}(r_{j_1}) \\ \vdots \\ -u_{(j_M)1}(r_{j_M}) \end{pmatrix} +
\beta_2^{(0)} \begin{pmatrix} u_{(i_1)2}(-r_{i_1}) \\ \vdots \\ u_{(i_L)2}(-r_{i_L}) \\ u_{(j_1)2}(r_{j_1}) \\ \vdots \\ u_{(j_M)2}(r_{j_M}) \end{pmatrix}= 0
\]
for solutions $u_{(i_p)}$ and $u_{(j_q)}$ of $Ju'=-zH_{i_p}u$ and $Ju'=-zH_{j_q}u$, respectively. 

Define the maximal relation of the canonical systems $Ju' = -zH_iu$ on the graph $G$ as
\begin{align*}
\mathcal T_G= \{ &(f,g) \in \left( \bigoplus_{i=1}^{k} L_{H_i}^2(-r_i,r_i) \right)^2 : f_{(i)} \textrm{ has an AC representative }\\
 &f_{(i)0} \textrm{ such that } Jf_{(i)0}'(x)=-H_ig_{(i)}(x) \textrm{ for a.e. } x\in (-r_i,r_i) \} . 
\end{align*}
Note that this is simply the direct sum of the maximal relations of the individual canonical systems on the edges. Define the relation $\mathcal S_G$ as the set of all $(f,g) \in \mathcal T_G$ such that the representatives $f_{(i)0}=h_{(i)}$ introduced in the definition of $\mathcal T_G$ satisfy
\[
\beta_1^{(0)} \begin{pmatrix} h_{(i_1)1}(-r_{i_1}) \\ \vdots \\ h_{(i_L)1}(-r_{i_L}) \\ -h_{(j_1)1}(r_{j_1}) \\ \vdots \\ -h_{(j_M)1}(r_{j_M}) \end{pmatrix} +
\beta_2^{(0)} \begin{pmatrix} h_{(i_1)2}(-r_{i_1}) \\ \vdots \\ h_{(i_L)2}(-r_{i_L}) \\ h_{(j_1)2}(r_{j_1}) \\ \vdots \\ h_{(j_M)2}(r_{j_M}) \end{pmatrix}= 0
\]
at every vertex $v$.

The canonical system on $G$ can be transformed into a higher order canonical system on $(0,1)$. The idea is to insert a vertex in the middle of each edge with Neumann-Kirchhoff interface conditions there. We then have $2k$ canonical systems, which, after changing the variables, are all on $(0,1)$ with $0$ corresponding to the newly inserted vertices and $1$ corresponding to the original vertices. These canonical systems can all be collected into one higher order system with a boundary condition at $0$ coming from the Neumann-Kirchhoff interface conditions at the inserted vertices and a boundary condition at $1$ determined by the original interface conditions.

To do this explicitly, let $N=\begin{pmatrix} -1 & 0 \\ 0 & 1 \end{pmatrix}$, and define
\begin{equation}
\label{splitH}
H_i^{(1)}(x)=r_i N H_i(-r_ix)N , \quad H_i^{(2)}(x)=r_iH_i(r_ix)
\end{equation}
for $x\in (0,1)$. For $f \in \bigoplus_{i=1}^{k} L_{H_i}^2(-r_i,r_i)$, set $f_{(i)}^{(1)}(x)=Nf_{(i)}(-r_ix)$ and $f_{(i)}^{(2)}(x)=f_{(i)}(r_ix)$. Note that $f_{(i)}^{(j)}(x)$, as an element of $L_{H_i^{(j)}}^2(0,1)$, does not depend on choice of the representative of $f_{(i)}\in  L_{H_i}^2(-r_i,r_i)$. Define 
\[
\tilde N: \bigoplus_{i=1}^{k} L_{H_i}^2(-r_i,r_i) \to \bigoplus_{j=1}^2 \bigoplus_{i=1}^{k} L_{H_i^{(j)}}^2(0,1) 
\]
by
\begin{align}
\label{tildeN}
(\tilde Nf)(x)&=\left(\left(f_{(i)}^{(1)}(x)\right)_{i=1}^k, \left(f_{(i)}^{(2)}(x)\right)_{i=1}^k\right) \nonumber \\
&=\left(\left(Nf_{(i)}(-r_ix)\right)_{i=1}^k, \left(f_{(i)}(r_ix)\right)_{i=1}^k\right) .
\end{align}
Clearly, $\tilde N$ is unitary. Consider the graph $\tilde G$ obtained from $G$ by inserting a vertex into the middle of each edge. Then the $H_i^{(1)}$ and $H_i^{(2)}$ form a canonical system on $\tilde G$. The maximal relation is defined by
\begin{align*}
\mathcal T_{\tilde G}= \{ &(f,g) \in \left( \bigoplus_{j=1}^2 \bigoplus_{i=1}^{k} L_{H_i^{(j)}}^2(0,1) \right)^2 : f_{(i)}^{(j)} \textrm{ has an AC representative }\\
 &f_{(i)0}^{(j)} \textrm{ such that } Jf_{(i)0}^{(j)'}(x)=-H_i^{(j)}g_{(i)}^{(j)}(x) \textrm{ for a.e. } x\in (0,1) \} . 
\end{align*}
Note that $\tilde N \mathcal T_G {\tilde N}^*$ is the subset of $\mathcal T_{\tilde G}$ defined by the condition that $f_{(i)0}^{(2)}(0)=Nf_{(i)0}^{(1)}(0)$.

The $2k$ canonical systems $H_i^{(1)}$ and $H_i^{(2)}$ can be rewritten as one canonical system $H(x) \in \C^{4k \times 4k}$. First, let 
\begin{equation}
\label{tildeH}
\tilde H(x) = \left( \bigoplus_{i=1}^k H_i^{(1)}(x) \right) \oplus \left( \bigoplus_{i=1}^k H_i^{(2)}(x) \right) ,
\end{equation}
and define a unitary matrix $C$ by 
\[
Ce_i=\begin{cases} e_{(i+1)/2} & i \textrm{ odd} \\ e_{2k+\frac{i}{2}} & i \textrm{ even} \end{cases} 
\]
for $1\le i \le 4k$.
Let 
\begin{equation}
\label{bigCS}
H=C\tilde{H}C^* .
\end{equation} 
Then $(f_{(i)}^{(j)}, g_{(i)}^{(j)}) \in \mathcal T_{H_i^{(j)}}$ for all $1\le i \le k$ and $j\in \{1,2\}$ if and only if $(f,g) \in \mathcal T_H$, where 
\[
f =C(f_{(1)}^{(1)t},\, \dots,\, f_{(k)}^{(1)t},\, f_{(1)}^{(2)t},\, \dots,\, f_{(k)}^{(2)t})^t 
\]
and
\[
g =C(g_{(1)}^{(1)t},\, \dots,\, g_{(k)}^{(1)t},\, g_{(1)}^{(2)t},\, \dots,\, g_{(k)}^{(2)t})^t .
\] 
So, the map $\tilde V: \bigoplus_{j=1}^2 \bigoplus_{i=1}^{k} L_{H_i^{(j)}}^2(0,1) \to L_H^2(0,1)$ defined by 
\begin{equation}
\label{tildeV}
\tilde V\left(\left(f_{(i)}^{(1)}\right)_{i=1}^k, \left(f_{(i)}^{(2)}\right)_{i=1}^k\right)=C(f_{(1)}^{(1)t},\, \dots,\, f_{(k)}^{(1)t},\, f_{(1)}^{(2)t},\, \dots,\, f_{(k)}^{(2)t})^t
\end{equation} 
is unitary and satisfies $\tilde V{\mathcal T_{\tilde G}}{\tilde V}^*=\mathcal T_H$, the maximal relation of the canonical system $H$.

The interface conditions $\beta^{(0)}(v)$ on the graph induce a boundary condition $\beta=(\beta_1 \: \beta_2)$ on the the canonical system $Ju'=-zHu$ at $x=1$. Define $\beta_1=(\tilde a_{ij})$ and $\beta_2=(\tilde b_{ij})$ as follows. Let $1 \le i \le k$. Note that $i$ can be written as $i=i_l(v)$ for a unique vertex $v$ (the notation is introduced in paragraph two of this section). Define   
\[
\tilde a_{ij}= \begin{cases} -a_{ij}(v) & j=i_m(v) \textrm{ or } j=j_q(v) \\ 0 & \textrm{ else} \end{cases} 
\]
and
\[
\tilde b_{ij}= \begin{cases} b_{ij}(v) & j=i_m(v) \textrm{ or } j=j_q(v) \\ 0 & \textrm{ else} \end{cases} .
\]
Suppose $k<i\le 2k$. So, $i-k=j_p(v)$ for a unique vertex $v$. Define 
\[
\tilde a_{ij}= \begin{cases} -a_{(i-k)j}(v) & j=i_m(v) \textrm{ or } j=j_q(v) \\ 0 & \textrm{ else} \end{cases} 
\]
and
\[
\tilde b_{ij}= \begin{cases} b_{(i-k)j}(v) & j=i_m(v) \textrm{ or } j=j_q(v) \\ 0 & \textrm{ else} \end{cases} .
\]
Note that $\beta \beta^{*}=I$ and $\beta J \beta^{*}=0$ due to the corresponding properties of $\beta^{(0)}(v)$. Let $\beta= (\beta_1 \: \beta_2)$. Define $\alpha_1=\frac{1}{\sqrt 2}\begin{pmatrix} I & I \\ 0 & 0 \end{pmatrix}$ and $\alpha_2=\frac{1}{\sqrt 2}\begin{pmatrix} 0 & 0 \\ I & -I \end{pmatrix}$, where $I \in \C^{k\times k}$. Introduce the boundary condition $\alpha=(\alpha_1 \: \alpha_2)$ at $0$. It is straightforward to check that $V{\mathcal S_G}V^*=\mathcal S_H^{\alpha, \beta}$, where $V=\tilde V \tilde N$, $\mathcal S_G$ is the self-adjoint relation introduced in the third paragraph of this section for the canonical systems on the graph, and $\mathcal S_H^{\alpha, \beta}$ is relation for the canonical system $H$ on $(0,1)$ defined as in \eqref{regcs}. So, the following theorem holds.

\begin{Theorem}
\label{compactGtoCS}
Let $G$ be a connected, compact graph with canonical systems $H_i$ defined along the edges, as above. The map $V=\tilde V \tilde N$ defined by \eqref{tildeN} and \eqref{tildeV} provides a unitary equivalence $V{\mathcal S_G}V^*=\mathcal S_H^{\alpha, \beta}$, where the canonical system $H$ on $(0,1)$ is defined by \eqref{splitH}, \eqref{tildeH}, and \eqref{bigCS}, with the boundary conditions $\alpha$ and $\beta$ described above. If $H$ is definite, then $\mathcal S_G$ is self-adjoint. In addition, if $\tilde G$ is the graph obtained from $G$ by inserting a vertex into the middle of each edge and the independent variables are changed so that $x\in (0,1)$ for each canonical system $H_i$, then $\tilde V{\mathcal T_{\tilde G}}{\tilde V}^*=\mathcal T_H$.
\end{Theorem}
\begin{remark}
As a consequence, the Green function, the resolvent, an $m$ function, and the associated spectral representation for $\mathcal S_G$ can be obtained via Theorems \ref{regm}, \ref{regres}, and \ref{regrep}, assuming that $H$ is definite. The problems become trivial when $H$ is not definite. See Proposition \ref{HGnotdef} and its proof.  
\end{remark}
\begin{remark}
By Corollary \ref{regHS}, the resolvent of $\mathcal S_G$ is Hilbert-Schmidt. In particular, the spectrum $\sigma=\{t_j\}$ is purely discrete and $\sum_j \frac{1}{1+t_j^2}<\infty$.
\end{remark}
\begin{remark}
Note that the unitary equivalence $V$ is obtained by a basic rewriting of the original equations on the graph. The $m$ function, the Green function, and the spectral representation are defined in terms of solutions of the equations $Ju'=-zH_iu$. It might also be of interest to consider the modified versions of Theorems \ref{regm}, \ref{regres}, and \ref{regrep} where, instead of starting with solutions $u$ and $v$ with initial conditions given at the left endpoint, which corresponds to the vertices inserted into the graph, one starts with solutions $u$ and $v$ with the initial conditions $u(b,z)=J\beta^*$ and $v(b,z)=\beta^*$ given at the right endpoint, which corresponds to the original vertices.   
\end{remark}
\begin{remark}
The theorem also provides a correspondence between self-adjoint interface conditions, not necessarily local, on the graph $G$ and boundary conditions for the canonical system $H$. The self-adjoint interface conditions on $G$ are in one-to-one correspondence with the self-adjoint boundary conditions $\beta$ at $x=1$ for $H$. The boundary conditions $\beta$ obtained via the above procedure from some matrices $\beta^{(0)}$ correspond exactly to the local interface conditions; boundary conditions $\beta$ that are not of this form correspond to non-local interface conditions on $G$. Note that although one could also consider different boundary conditions $\alpha$ at $x=0$ or coupled boundary conditions for $H$, these would correspond to changes of interface conditions among the inserted vertices in $\tilde G$ or coupling between the inserted vertices and the original vertices.      
\end{remark}
\begin{remark}
The choice that the graph had canonical systems $H_i(x)\in\R^{2\times 2}$ was made for purely cultural reasons. All of the above procedure can be easily generalized to any canonical systems $H_i(x)\in\C^{2n\times 2n}$ on a graph. One could even have different orders on different edges.
\end{remark}

\begin{Proposition}
\label{HGnotdef}
Let $G$ be a compact graph with canonical systems $H_i$ on $G$, as in the first paragraph of this section, and define $H$ by \eqref{splitH}, \eqref{tildeH}, and \eqref{bigCS}. If $H$ is not definite, then $D(\mathcal T_G)$, the domain of the maximal relation for the canonical systems $H_i$ on $G$, is finite-dimensional.
\end{Proposition}
\begin{proof}
$H$ is not definite if and only if each $H_i^{(j)}$ is of the form $H_i^{(j)}(x)=h_{ij}(x)P_{\theta_{ij}}$, where $P_\theta= \begin{pmatrix} \cos^2\theta & \sin\theta\cos\theta \\ \sin\theta\cos\theta & \sin^2\theta \end{pmatrix}$ \cite[Section 1.2]{Rembook}. So, it suffices to show that $D(\mathcal T_{H_\theta})$ is finite-dimensional for any $H_{\theta}(x)=h(x)P_{\theta}$ on $(0,1)$. Suppose $(f,g) \in \mathcal T_{H_\theta}$. Let $f_0$ be an absolutely continuous representative of $f$ such that $Jf_0'=-H_{\theta}g$. Since $f_0'(x)=h(x)JP_\theta g(x)$, $P_\theta f_0'(x)=0$. Hence, $P_\theta f_0(x)=c\begin{pmatrix} \cos\theta \\ \sin\theta \end{pmatrix}$ for some constant $c \in \C$. This means that $f$ and $c\begin{pmatrix} \cos\theta \\ \sin\theta \end{pmatrix}$ represent the same element in $L_{H_\theta}^2(0,1)$. Hence, $D(\mathcal T_{H_\theta})= \{ c\begin{pmatrix} \cos\theta \\ \sin\theta \end{pmatrix} : c\in \C \}$.
\end{proof}
\begin{remark}
In this context, it follows that the spectrum consists of a finite number, possibly zero, of eigenvalues if $H$ is not definite. The eigenvalue equations can be solved easily, and, not surprisingly, checking the conditions that the solutions satisfy the interface conditions and do not represent the zero element in the Hilbert space reduces to a finite-dimensional problem in linear algebra.
\end{remark} 

\section{Canonical systems on non-compact graphs}
Consider now a canonical system on a non-compact graph. Again, assume that the graph $G$ is connected, has finitely many edges and vertices, and that there is at most one edge between any two vertices. Some edges have only one vertex; these are pictured as half lines. For all other edges, assume that they connect two distinct vertices. Label the edges $E_1,\, \dots,\, E_k$, with $E_1, \, \dots, \, E_{\tilde k}$ being the edges with two vertices. With each edge $E_i$, $1\le i \le \tilde k$, associate a differential equation $Ju'(x) = -zH_i(x)u(x)$ with $x\in (-r_i,r_i)$, $H_i(x)\in\R^{2\times 2}$, $H_i \in L^1(-r_i,r_i)$, and $H_i(x)\ge 0$. For each edge $E_i$, $\tilde k < i \le k$, associate a canonical system $Ju'(x) = -zH_i(x)u(x)$ with $x\in (-1,\infty)$, $H_i(x)\in\R^{2\times 2}$, $-1$ a regular endpoint, and $H_i \notin L^1(-1,\infty)$. Recall from Section 2 that the canonical systems are always assumed to be locally integrable.

It can be assumed, without loss of generality, that $H_i$ is definite on $(0,\infty)$ for all $\tilde k <i\le k$. To see this, suppose that $H_i$, $i>\tilde k$, is not definite on $(0,\infty)$. Then there exist $\theta \in \R$ and $h: (0,\infty) \to [0,\infty)$ such that, for almost all $x>0$, $H_i(x)=h(x)P_\theta$, where $P_\theta= \begin{pmatrix} \cos^2\theta & \sin\theta\cos\theta \\ \sin\theta\cos\theta & \sin^2\theta \end{pmatrix}$ and $h \notin L^1(0,\infty)$ \cite[Section 1.2]{Rembook}. Suppose $Jf'=-H_ig$ with $f,g \in L_{H_i}^2(-1,\infty)$ and $f$ absolutely continuous. So, $f'=hJPg$ on $(0,\infty)$. Since $PJP=0$, $Pf(x)=Pf(0)$ for all $x>0$. So, $\int_0^\infty h(x)f(0)^*Pf(0) \, dx$ is finite. This is only possible if $Pf(0)=0$, which also implies that $f\big |_{(0,\infty)}=0 \in L_{H_i}^2(0,\infty)$. Thus, the half line $(0,\infty)$ could be replaced by a vertex at $0$ with the interface condition $(\cos\theta \: \sin\theta)f(0)=0$. So, assume that $H_i$ is definite on $(0,\infty)$ for all $\tilde k <i\le k$.

Introduce interface conditions at the vertices in the same way as in paragraph two of the previous section. The only modification happens at the vertices connected to at least one half line. Obviously, when evaluating functions to check whether they satisfy the interface conditions at such a vertex, the functions on the half-lines are evaluated at $x=-1$; such a vertex $v$ is defined as the initial vertex of such an edge. The same notation for the interface conditions from the previous section will be used in this section. Also, make the obvious modifications to define the maximal relation $\mathcal T_G$ and $\mathcal S_G$. No conditions at $\infty$ (other than being in $L_{H_i}^2$) for the functions on the half lines are imposed in the definition of $\mathcal S_G$; $\mathcal S_G$ is self-adjoint, which will be proved below.

Again, such a problem can be rewritten as a single, higher order canonical system. The idea is to combine the techniques of the previous section with the following theorem, which is well-known in the case $n=1$ \cite{Rembook}.
\begin{Theorem}
\label{singhalfline}
Let $H(x) \in \C^{2n \times 2n}$ be a canonical system on $(0,\infty)$, with $0$ a regular endpoint. Let $\alpha=(\alpha_1 \: \alpha_2)\in \C^{n\times 2n}$ and $\beta=(\beta_1 \: \beta_2) \in \C^{n\times 2n}$ satisfy \eqref{bc}. Suppose that $H(x)=\begin{pmatrix} \beta_1^*\beta_1 & \beta_1^*\beta_2 \\ \beta_2^*\beta_1 & \beta_2^*\beta_2 \end{pmatrix}$ for $x>L$. Let 
\begin{align*}
\mathcal S= \{ &(f,g) \in \left( L_H^2(0,\infty) \right)^2 : f \textrm{ has an AC representative } f_0 \textrm{ such that } \\
 &Jf_0'(x)=-H(x)g(x) \textrm{ for a.e. } x\in (0,\infty) \textrm{ and } \alpha f_0(0)=0 \} .
\end{align*}
Then $\mathcal S= \mathcal S^{\alpha, \beta}(0,L) \oplus (0, L_H^2(L,\infty))$. If $H$ is definite on $(0,L)$, then the deficiency indices of $H$ on $(0,\infty)$ are both $n$,  $S=S^{\alpha,\beta}$, and the $m$ functions for $\mathcal S$ and $\mathcal S^{\alpha, \beta}(0,L)$ are the same.
\end{Theorem} 
\begin{proof}
Let $P=\begin{pmatrix} \beta_1^*\beta_1 & \beta_1^*\beta_2 \\ \beta_2^*\beta_1 & \beta_2^*\beta_2 \end{pmatrix}$. Suppose that $(f,g) \in \mathcal S$, and take an absolutely continuous representative $f_0$ of $f$ such that $\alpha f_0(0)=0$ and $Jf_0'=-Hg$ on $(0,\infty)$. So, since $Jf_0'=-Pg$ on $(L,\infty)$ and $PJP=0$, $Pf_0(x)=Pf_0(L)$ for all $x>L$. Since $\int_L^\infty f_0(L)^*Pf_0(L) \, dx=\int_L^\infty f(x)^*Pf(x) \, dx < \infty$, 
$Pf_0(L)=0$. This shows that $(f\big |_{(0,L)}, g\big |_{(0,L)}) \in \mathcal S^{\alpha, \beta}(0,L)$ and $Pf=0$ almost everywhere on $(L,\infty)$.

Now, suppose that $(\tilde f, \tilde g) \in \mathcal S^{\alpha, \beta}(0,L)$. Take an absolutely continuous representative $\tilde f_0$ of $f$ such that $\alpha \tilde f_0(0)=0=\beta \tilde f_0(L)$ and $J\tilde f_0'=-H\tilde g$ on $(0,L)$. Define
\[
(f(x),g(x))=\begin{cases} (\tilde f_0(x),\tilde g(x)) & x \le L\\ (f_0(L),0) & x> L \end{cases} .
\]
Since $P\tilde f_0(L)=0$, $(f,g) \in \mathcal S$. 

Let $(0, \tilde g) \in (0, L_H^2(L,\infty))$ be given. Define
\[
g(x)=\begin{cases} 0 & x \le L\\ \tilde g(x)) & x> L \end{cases} 
\]
and $f(x)= \int_0^x JH(y)g(y) \, dy$. Then $Jf'=-Hg$, $f(x)=0$ for $x\in [0,L]$, and $f(x)= \int_L^x JPg(y) \, dy$ for $x>L$. So, since $PJP=0$, $f$ represents $0$ in $L_H^2(0,\infty)$. Since also $\alpha f(0)=0$, $(f,g) \in \mathcal S$, and $g \in \mathcal S(0)$.

Suppose that $H$ is definite on $(0,L)$. The argument in the first paragraph shows that solutions of $Ju'=-zHu$ are in $L_H^2(0,\infty)$ if and only they satisfy the boundary condition at $L$. Since $H$ is integrable on $(0,L)$, there are $2n$ linearly independent solutions of $Ju'=-iHu$ in $L_H^2(0,L)$ (and of $Ju'=iHu$). Hence, there $n$ linearly independent solutions in $L_H^2(0,L)$ that satisfy the boundary condition at $L$. So, there are (exactly) $n$ linearly independent solutions in $L_H^2(0,\infty)$. Since $H$ is definite on $(0,\infty)$, this means that the deficiency indices are $n$. In particular, $\mathcal S$ is self-adjoint. So, since $L_H^2(L,\infty)\subset \mathcal S(0)$, $S=S^{\alpha,\beta}$.

Let $m_{\alpha}$ and $m_{\alpha,\beta}$ be the $m$ function for $\mathcal S$ and $\mathcal S^{\alpha, \beta}(0,L)$, respectively. Let $u$ and $v$ be as in the definition of the $m$ functions in Theorems \ref{regm} and \ref{lpm}. Fix $z \in \C/\R$. So, $m_{\alpha}(z)$ is the unique matrix such that $f(x,z)=v(x,z)+u(x,z)m_{\alpha}(z)$ is in $L_H^2(0,\infty)$, and $m_{\alpha,\beta}(z)$ is the unique matrix such that $\tilde f(x,z)=v(x,z)+u(x,z)m_{\alpha,\beta}(z)$ satisfies $\beta \tilde f(L,z)=0$. Since $Jf'=-zHf$ and $f\in L_H^2(0,\infty)$, the argument in the first paragraph shows that $Pf(L,z)=0$, which is equivalent to $\beta f(L,z)=0$. By the uniqueness of $m_{\alpha,\beta}(z)$, $m_{\alpha}(z)=m_{\alpha,\beta}(z)$. 
\end{proof}

In other words, a problem on a bounded interval with a boundary condition $\beta$ at the right endpoint is equivalent to a problem on $(0,\infty)$ that ends with a singular half line, where $H$ is identically $\begin{pmatrix} \beta_1^*\beta_1 & \beta_1^*\beta_2 \\ \beta_2^*\beta_1 & \beta_2^*\beta_2 \end{pmatrix}$. One immediate consequence of this and Theorem \ref{compactGtoCS} is that canonical systems on compact graphs can be rewritten as canonical systems on $(0,\infty)$. Strictly speaking, Theorem \ref{singhalfline} is not used here to deal with non-compact graphs, but the ideas in the theorem and its proof provide the motivation for some of the development below.

The outline of the argument that a problem on a non-compact graph can be converted to a higher order system is as follows. Insert a vertex into the middle of each edge with finite length and at $x=0$ on the edges with infinite length. Call the resulting graph $\tilde G$. Change variables on all of the finite length edges in $\tilde G$ so that $0$ corresponds to the inserted vertices and $1$ to the original vertices. Impose Neumann-Kirchhoff interface conditions at all of the inserted vertices. Collect all of the canonical systems on the finite edges, including the ones coming from the initial intervals of the half lines, into one canonical system $H_0$ and rewrite the original interface conditions in terms of a boundary condition $\beta$ at $1$ for $H_0$, as in the previous section. One can then combine the canonical systems on the half lines $(0,\infty)$ with $H_0$ on $(0,1)$ and with $\begin{pmatrix} \beta_1^*\beta_1 & \beta_1^*\beta_2 \\ \beta_2^*\beta_1 & \beta_2^*\beta_2 \end{pmatrix}$ on $(1,\infty)$ into one canonical system with a boundary condition $\alpha$ at $0$ determined by the Neumann-Kirchhoff interface conditions at the inserted vertices.

More carefully, let $N=\begin{pmatrix} -1 & 0 \\ 0 & 1 \end{pmatrix}$ as before, and define $r_i=1$ for $\tilde k < i \le k$. Set
\begin{equation}
\label{splitH1}
H_i^{(1)}(x)=r_i N H_i(-r_ix)N
\end{equation}
for $x\in (0,1)$ and
\begin{equation}
\label{splitH2}
H_i^{(2)}(x)=r_iH_i(r_ix)
\end{equation}
for $x\in (0,1)$ if $1\le i \le \tilde k$  and for $x\in (0,\infty)$ if $\tilde k < i \le k$. Again, the map $\tilde N: \bigoplus_{i=1}^{k} L_{H_i}^2 \to \bigoplus_{j=1}^2 \bigoplus_{i=1}^{k} L_{H_i^{(j)}}^2$ defined by 
\begin{equation}
\label{tildeNlp}
\tilde Nf=\left(\left(f_{(i)}^{(1)}\right)_{i=1}^k, \left(f_{(i)}^{(2)}\right)_{i=1}^k\right) ,
\end{equation}
where $f_{(i)}^{(1)}(x)=Nf_{(i)}(-r_ix)$ and $f_{(i)}^{(2)}(x)=f_{(i)}(r_ix))$, is unitary and $\tilde N \mathcal T_G {\tilde N}^*$ is the subset of $\mathcal T_{\tilde G}$ defined by the condition that $Nf_{(i)0}^{(1)}(0)=f_{(i)0}^{(2)}(0)$.

The matrix $\beta=(\beta_1 \: \beta_2)$ encoding the interface conditions $\beta^{(0)}(v)$ on the graph is defined in essentially the same way as in the previous section. Let $\beta_1=(\tilde a_{ij})\in \C^{(k+\tilde k)\times (k+\tilde k)}$ and $\beta_2=(\tilde b_{ij})\in \C^{(k+\tilde k)\times (k+\tilde k)}$. Let $1 \le i \le k$. Then $i$ can be written as $i=i_l(v)$ for a unique vertex $v$ (see paragraph two of the previous section for the notation). Set   
\[
\tilde a_{ij}= \begin{cases} -a_{ij}(v) & j=i_m(v) \textrm{ or } j=j_q(v) \\ 0 & \textrm{ else} \end{cases} 
\]
and
\[
\tilde b_{ij}= \begin{cases} b_{ij}(v) & j=i_m(v) \textrm{ or } j=j_q(v) \\ 0 & \textrm{ else} \end{cases} .
\]
Suppose $k<i \le k+\tilde k$. Then $i-k=j_p(v)$ for a unique vertex $v$. Define 
\[
\tilde a_{ij}= \begin{cases} -a_{(i-k)j}(v) & j=i_m(v) \textrm{ or } j=j_q(v) \\ 0 & \textrm{ else} \end{cases} 
\]
and
\[
\tilde b_{ij}= \begin{cases} b_{(i-k)j}(v) & j=i_m(v) \textrm{ or } j=j_q(v) \\ 0 & \textrm{ else} \end{cases} .
\]

For $x \in (0,1)$ set 
\begin{equation}
\label{tildeHa}
\tilde H(x) = \left( \bigoplus_{i=1}^k H_i^{(1)}(x) \right) \oplus \left( \bigoplus_{i=1}^k H_i^{(2)}(x) \right) ,
\end{equation}
and let $C_d\in \C^{d \times d}$, for even $d$, be the matrix defined by 
\[
C_de_i=\begin{cases} e_{(i+1)/2} & i \textrm{ odd} \\ e_{\frac{d+i}{2}} & i \textrm{ even} \end{cases} 
\]
for $1\le i \le d$.
Define
\begin{equation}
\label{bigCSa}
H(x)=C_{4k}\tilde{H}(x)C_{4k}^* 
\end{equation} 
for $x\in (0,1)$.  To define $H$ on $(1,\infty)$, first define another change of basis $D$ by $De_{2k-i}=e_{3k+\tilde k -i}$ for $0\le i \le k-\tilde k-1$, $De_{2k+i}=e_{k+\tilde k+i}$ for $1<i \le k+\tilde k$, and $De_i=e_i$ for $1\le i \le k+\tilde k$ and $3k+2\tilde k<i\le 4k$. In other words, $D$ shifts the entries corresponding to the first components of solutions on the half lines $(0,\infty)$ down so that they collectively sit above the second components. Then set, for $x\in (1,\infty)$, 
\begin{equation}
\label{bigCSb}
H(x)=D^*\left( \begin{pmatrix} \beta_1^*\beta_1 & \beta_1^*\beta_2 \\ \beta_2^*\beta_1 & \beta_2^*\beta_2 \end{pmatrix} \oplus \left({C_{2(k-\tilde k)}}\bigoplus_{i=\tilde k +1}^{k}H_i^{(2)}(x){C^*_{2(k-\tilde k)}}\right)\right)D .
\end{equation}

Let $Q: \C^{4k} \to \C^{2k+2\tilde k}$ be the map such that, for all $v\in \C^{4k}$, $Qv$ deletes the entries $v_i$ for $k+\tilde k +1\le i \le 2k$ and $3k+2\tilde k <i \le 4k$, the entries corresponding to the half lines $(0,\infty)$, and let $Q_\perp : \C^{4n} \to \C^{2k-2\tilde k}$ be the map such that $Q_\perp v$ deletes the entries $v_i$ for $1\le i \le k+\tilde k$ and $2k< i \le 3k+2\tilde k$.

\begin{Proposition}
\label{singhalfline2}
Let $H$ be defined by \eqref{bigCSa} and \eqref{bigCSb}, and consider the relation
\begin{align*}
\mathcal S_H= \{ &(f,g) \in \left( L_H^2(0,\infty) \right)^2 : f \textrm{ has an AC representative } f_0 \textrm{ such that } \\
 &Jf_0'(x)=-H(x)g(x) \textrm{ for a.e. } x\in (0,\infty) \textrm{ and } \alpha f_0(0)=0 \} ,
\end{align*}
where $\alpha_1=\frac{1}{\sqrt 2}\begin{pmatrix} I & I \\ 0 & 0 \end{pmatrix}$ and $\alpha_2=\frac{1}{\sqrt 2}\begin{pmatrix} 0 & 0 \\ I & -I \end{pmatrix}$ with $I \in \C^{k\times k}$. Set 
\begin{align*}
\mathcal S_1= \{ &(f,g) \in \mathcal S_H : PQf_0(x)=0 \textrm { for all } x\ge 1 \textrm{ and} \\ 
&Qg \big |_{(1,\infty)}=0 \in L_P^2(1,\infty) \} ,
\end{align*}
where $P=\begin{pmatrix} \beta_1^*\beta_1 & \beta_1^*\beta_2 \\ \beta_2^*\beta_1 & \beta_2^*\beta_2 \end{pmatrix}$, and
\begin{align*}
\mathcal S_2= \{ &(0,g) \in \left( L_H^2(0,\infty) \right)^2 : Q_\perp g \big |_{(1,\infty)}=0 \in L_{H_\infty}^2(1,\infty) \textrm { and}\\
 & g\big |_{(0,1)}=0 \in L_H^2(0,1) \} ,
\end{align*}
where $H_{\infty}(x)={C_{2(k-\tilde k)}}\bigoplus_{i=\tilde k +1}^{k}H_i^{(2)}(x){C^*_{2(k-\tilde k)}}$. Then $\mathcal S_H=\mathcal S_1 \oplus \mathcal S_2$. 
\end{Proposition}
\begin{proof}
First, to check that $\mathcal S_2 \subset \mathcal S_H$, let $(0,g)\in \mathcal S_2$ be given, and define $f(x)=\int_0^x JH(y)g(y) \, dy$. Then $f(x)=0$ for $0\le x \le 1$ since $g\big |_{(0,1)}=0 \in L_H^2(0,1)$. For $x>1$, $Q_\perp f(x)=0$ since $Q_\perp g \big |_{(1,\infty)}=0 \in L_{H_\infty}^2(1,\infty)$ and $f(1)=0$. Since $PJP=0$, $Qf\big |_{(1,\infty)}=0\in L_P^2(1,\infty)$. Hence, $f=0\in L_H^2(0,\infty)$, $f$ is absolutely continuous, $Jf'=-Hg$, and $\alpha f(0)=0$. So, $(0,g)=(f,g)\in \mathcal S_H$.

Now, let $(f,g)\in \mathcal S_H$ be given. Define $f_1$ by setting $f_1(x)=f_0(x)$ for $x\in (0,1)$, $Qf_1(x)=Qf_0(1)$ for $x\in [1,\infty)$, and $Q_\perp f_1(x)=Q_\perp f_0(x)$ for $x\in [1,\infty)$. Define $g_1$ by setting $g_1(x)=g(x)$ for $x\in (0,1]$, $Qg_1(x)=0$ for $x\in (1,\infty)$, and $Q_\perp g_1(x)=Q_\perp g(x)$ for $x\in(1,\infty)$. Since $JQf_0'=-PQg$ on $(1,\infty)$ and $PJP=0$, $PQf_0(x)=PQf_0(1)$ for all $x>1$. So, since $\int_1^\infty (Qf_0(1))^*PQf_0(1) \, dx< \infty$, $PQf_0(1)=0$. Hence, $f_1\in L_H^2(0,\infty)$. It is obvious that $f_1$ is absolutely continuous, $Jf_1'=-Hg_1$, and $\alpha f_1(0)=0$. Hence $(f_1,g_1) \in \mathcal S_H$. Also, since $PQf_1(x)=PQf_0(1)=0$ and $Qg_1(x)=0$ for $x>1$, $(f_1,g_1) \in \mathcal S_1$.

Consider $(f_2,g_2)=(f,g)-(f_1,g_1)$. For $x\in(0,1)$, $H(x)f_2(x)=H(x)(f(x)-f_0(x))=0$ almost everywhere. For $x>1$, $PQf_2(x)=PQ(f(x)-f_0(1))=0$ almost everywhere since $PQf_0(x)=PQf_0(1)$ for all $x>1$. Also, since $Q_\perp f_1(x)=Q_\perp f_0(x)$ for $x\in(1,\infty)$, $H_\infty(x)Q_\perp f_2(x)=H_\infty(x)Q_\perp(f(x)-f_0(x))=0$ for almost every $x>1$. Hence, $f_2=0\in L_H^2(0,\infty)$. Since $g_1(x)=g(x)$ for $x\in (0,1)$ and $Q_\perp g_1(x)=Q_\perp g(x)$ for $x\in(1,\infty)$, $g_2\big |_{(0,1)}=0$ and $Q_\perp g_2 \big |_{(1,\infty)}=0$. So, $(f_2,g_2)\in \mathcal S_2$.

By construction, $\langle (f_1,g_1), (0,g_2)\rangle=$ for any $(f_1,g_1) \in \mathcal S_1$ and $(0,g_2) \in \mathcal S_2$. So, since $\mathcal S_H=\mathcal S_1+ \mathcal S_2$ by the above, the claim follows.
\end{proof}

\begin{Proposition}
The deficiency indices of $H$ are both $2k$. In particular, $\mathcal S_H$ is self-adjoint. 
\end{Proposition}
\begin{proof}
Since $H$ is definite, due to the assumption that $H_i$ is definite on $(0,\infty)$ for $i>\tilde k$, it suffices to show that there are (exactly) $2k$ linearly independent solutions of $Ju'=-iHu$ in $L_H^2(0,\infty)$, and similarly for $Ju'=iHu$ \cite[Proposition 2.19]{Lesch}. Since the following argument works equally well for $Ju'=iHu$, just consider $Ju'=-iHu$. Let
\[
H_0(x)= \begin{cases} C_{2(k+\tilde k)}\left( \left( \bigoplus_{i=1}^k H_i^{(1)}(x) \right) \oplus \left( \bigoplus_{i=1}^{\tilde k} H_i^{(2)}(x) \right) \right) {C^*_{2(k+\tilde k)}} & x<1 \\ P & x>1 \end{cases} .
\]
So, it suffices to show that there are $k+\tilde k$ linearly independent solutions of $Ju'=-iH_0u$ in $L_{H_0}^2(0,\infty)$ and $k-\tilde k$ linearly independent solutions of $Ju'=-iH_\infty u$ in $L_{H_\infty}^2(0,\infty)$. Note that the linear independence being checked is that of functions and this is sufficient since $H$ is definite. 

Suppose $Ju'=-iH_0u$ and $u\in L_{H_0}^2(0,\infty)$. Then $Pu'=iPJPu=0$ on $(1,\infty)$. So, $Pu(x)=Pu(1)$ for $x>1$. Since $u\in L_{H_0}^2(0,\infty)$, this implies that $\beta u(1)=0$. Conversely, if $Ju'=-iH_0u$ on $(0,1)$ and $\beta u(1)=0$, then $u$ extends to a solution in $L_{H_0}^2(0,\infty)$. Now, since $H_0$ is integrable on $(0,1)$, there are $2(k+\tilde k)$ linearly independent solutions of $Ju'=-iH_0u$ in $L_{H_0}^2(0,1)$. Hence, there are $k+\tilde k$ linearly independent solutions of $Ju'=-iH_0u$ in $L_{H_0}^2(0,1)$ that satisfy $\beta u(1)=0$. So, there are there are $k+\tilde k$ linearly independent solutions of $Ju'=-iH_0u$ in $L_{H_0}^2(0,\infty)$.

Let $\tilde k<i\le k$, Since $H_i$ is definite on $(0,\infty)$ by assumption, there is a unique nontrivial solution of $Ju'=-iH_iu$ in $L_{H_i}^2(0,\infty)$. Hence, there are $k-\tilde k$ linearly independent solutions of $Ju'=-iH_\infty u$ in $L_{H_\infty}^2(0,\infty)$.  

Note that $0$ is a regular endpoint of $H$, $H$ is definite on $(0,\infty)$, $H(x)\in \C^{4k\times 4k}$, and both deficiency indices are $2k$. So, by the references cited in Section 2, $\mathcal S_H$ is self-adjoint.
\end{proof}

We now set up a map $V: \overline{D(\mathcal S_G)} \to \overline{D(\mathcal S_H)}$ that provides a unitary equivalence between the operators parts of $\mathcal S_H$ and $\mathcal S_G$, the relation for the problem on the graph introduced in the third paragraph of this section. For $f\in \bigoplus_{i=1}^{k} L_{H_i}^2$, let $\left(\left(f_{(i)}^{(1)}\right)_{i=1}^k,\, \left(f_{(i)}^{(2)}\right)_{i=1}^k\right)$ denote $\tilde Nf$, as in \eqref{tildeNlp}. Define a map $V_0$ on $\bigoplus_{i=1}^{k} L_{H_i}^2$ by 
{%
\small
\begin{equation}
\label{V0}
(V_0f)(x)=\begin{cases} C_{4k}\left(f_{(1)}^{(1)}(x)^t,\, \dots,\, f_{(k)}^{(1)}(x)^t,\, f_{(1)}^{(2)}(x)^t,\, \dots,\, f_{(k)}^{(2)}(x)^t\right)^t  & x<1 \\ 
C_{4k}\left(0,\, \dots,\, 0,\, f_{(\tilde k+1)}^{(2)}(x)^t,\, \dots,\, f_{(k)}^{(2)}(x)^t\right)^t & x \ge 1 \end{cases} ,
\end{equation}
}%
which is a well-defined isometry onto some subspace of $L_H^2(0,\infty)$; it is well-defined in the sense that the formula on the right above determines the same element of $L_H^2(0,\infty)$ for every representative of $f$ in $\bigoplus_{i=1}^{k} L_{H_i}^2$. Take any $(f,g)\in \mathcal S_G$ and absolutely continuous representatives $f_{(i)0}^{(j)}$ of $f_{(i)}^{(j)}$ that satisfy the interface conditions on $G$ and the equations $Jf_{(i)0}^{(j)'}=-H_i^{(j)}g_{(i)}^{(j)}$. Define $\tilde f$ by
\[
\tilde f(x)=C_{4k}\left(f_{(1)0}^{(1)}(x)^t,\, \dots,\, f_{(k)0}^{(1)}(x)^t,\, f_{(1)0}^{(2)}(x)^t,\, \dots,\, f_{(k)0}^{(2)}(x)^t\right)^t 
\]
for $x\in (0,1)$ and
\[
\tilde f(x)=C_{4k}\left(f_{(1)0}^{(1)}(1)^t,\, \dots,\, f_{(\tilde k)0}^{(2)}(1)^t,\, f_{(\tilde k +1)0}^{(2)}(x)^t,\, \dots,\, f_{(k)}^{(2)}(x)^t\right)^t 
\]
for $x\ge 1$. Since the functions $f_{(i)0}^{(j)}$ satisfy the interface conditions, $PQ\tilde f(x)=0$ for all $x\ge 1$. Hence, $\tilde f$ is a representative of $V_0f$. Also, $\tilde f$ is absolutely continuous, $J\tilde f'=-HV_0g$, and $\alpha \tilde f(0)=0$. Now, by Proposition \ref{singhalfline2}, every $\tilde f \in D(\mathcal S_H)$ is equal to $V_0f$ in $L_H^2(0,\infty)$ for some $f\in D(\mathcal S_G)$. So, $V_0$ restricts to a unitary map $V:\overline{D(\mathcal S_G)} \to \overline{D(\mathcal S_H)}$. In particular, for $(f,g)\in \mathcal S_G$ with $g\in \overline{D(\mathcal S_G)}$, $Vg=S_H\tilde f=S_HVf$ since $S_H$ selects the image in $\overline{D(\mathcal S_H)}$. 

Next, we want to show that $\mathcal S_G(0)=\overline{D(\mathcal S_G)}^\perp$. Note that $\mathcal S_H(0)=\overline{D(\mathcal S_H)}^\perp$ since $\mathcal S_H$ is self-adjoint, but we do not yet know whether $\mathcal S_G$ is self-adjoint. Suppose $(f,0)=(0,g)\in \mathcal S_G$, and define $\tilde f$ as above. So, $(\tilde f, V_0g)\in \mathcal S_H$. Since $V_0$ is an isometry, $\tilde f=0 \in L_H^2(0,\infty)$. Thus, $V_0g\in \overline{D(\mathcal S_H)}^\perp$. Since $V_0$ is an isometry and $V(\overline{D(\mathcal S_G)})=\overline{D(\mathcal S_H)}$, it follows that $g\in \overline{D(\mathcal S_G)}^\perp$. So, $\mathcal S_G(0) \subset \overline{D(\mathcal S_G)}^\perp$.

Let $g\in \overline{D(\mathcal S_G)}^\perp$. Then $V_0g\in \overline{D(\mathcal S_H)}^\perp=\mathcal S_H(0)$. So, there exists an absolutely continuous $\tilde f$ such that $\tilde f=0\in L_H^2(0,\infty)$, $J\tilde f'=-HV_0g$, and $\alpha \tilde f(0)=0$. Now, 
\[
\tilde f(x)=C_{4k}\left(f_{(1)}^{(1)}(x)^t,\, \dots,\, f_{(k)}^{(1)}(x)^t,\, f_{(1)}^{(2)}(x)^t,\, \dots,\, f_{(k)}^{(2)}(x)^t\right)^t
\]
for some functions $f_{(i)}^{(j)}$, and by setting 
\[
f=\tilde N^{-1}\left(\left(f_{(i)}^{(1)}\big |_{(0,1)}\right)_{i=1}^k,\, \left(f_{(i)}^{(2)}\big |_{(0,1)}\right)_{i=1}^{\tilde k}, \, \left(f_{(i)}^{(2)}\right)_{i=\tilde k+1}^k \right) ,
\]
one obtains $(f,g) \in \mathcal S_G$ such that $V_0f=\tilde f$ in $L_H^2(0,\infty)$. Since $\tilde f=0\in L_H^2(0,\infty)$ and $V_0$ is an isometry, $f=0\in \bigoplus_{i=1}^{k} L_{H_i}^2$. Hence, $g\in \mathcal S_G(0)$. Thus, $\overline{D(\mathcal S_G)}^\perp \subset \mathcal S_G(0)$. This completes the proof of the following theorem.

\begin{Theorem}
\label{GtoCS}
The isometry $V_0$ given by \eqref{V0} restricts to a unitary map $V: \overline{D(\mathcal S_G)} \to \overline{D(\mathcal S_H)}$. Let $S_G=\mathcal S_G\cap {\overline{D(\mathcal S_G)}}\oplus\overline{D(\mathcal S_G)}$. Then $VS_GV^*=S_H$. Moreover, $\mathcal S_G(0)= \overline{D(\mathcal S_G)}^\perp$ and $V_0 \mathcal S_G(0) \subset \mathcal S_H(0)$.
\end{Theorem}
\begin{remark}
It follows that $\mathcal S_G$ is a self-adjoint relation and $S_G$ is a self-adjoint operator. Since $\mathcal S_G(0)=\overline{D(\mathcal S_G)}^\perp$ is the kernel of both $(\mathcal S_G -z)^{-1}$ and any spectral representation of $\mathcal S_G$, the resolvent of $\mathcal S_G$ and a spectral representation for $\mathcal S_G$ can be obtained by combining Theorems \ref{GtoCS}, \ref{lpres}, and \ref{lprep}. Note that $H$ is definite on $(0,\infty)$ due the assumption that $H_i$ is definite on $(0,\infty)$ for $\tilde k<i\le k$.
\end{remark}

\section{Schr\"odinger operators on graphs}
A Schr\"odinger operator on a graph is defined in much the same way as a canonical system on a graph. The graph $G$ is assumed here to be connected, have finitely many edges and vertices, and have at most one edge between any two vertices. Some edges, corresponding to leads, may have only one vertex. For all other edges, again assume that they connect distinct vertices. Label the edges $E_1,\, \dots,\, E_k$, with $E_1, \, \dots, \, E_{\tilde k}$ being the edges with two vertices. With each edge $E_i$, $1\le i \le \tilde k$, associate a differential equation $-y''(x)+V_i(x)y(x)=zy(x)$ with $x\in (-r_i,r_i)$, $V_i(x)\in\R$, and $V_i \in L^1(-r_i,r_i)$. For each edge $E_i$, $\tilde k < i \le k$, if any, associate a differential equation $-y''(x)+V_i(x)y(x)=zy(x)$ with $x\in (-1,\infty)$, $V_i(x)\in\R$, $V_i$ integrable near $-1$, and $V_i\in L_{\textrm{loc}}^1(-1,\infty)$. For convenience, assume that the equations on the half lines are in the limit point case at infinity.
 
Suppose the graph is directed; for any vertex $v$ connected to a half line, $v$ is defined to be the initial vertex of that edge. Introduce interface conditions at the vertices in the following way. Take any vertex $v$. Define $E_o(v)= \{ i : v \textrm{ \rm{is the initial vertex of} }E_i \}$ and $E_t(v)= \{ i : v \textrm{ \rm{is the terminal vertex of} }E_i \}$. Label the numbers in $E_o(v)$ as $i_1<\cdots<i_L$ , $i_p=i_p(v), \, L=L(v)$, and the numbers in $E_t(v)$ as $j_1<\cdots<j_M$, $j_p=j_p(v),\, M=M(v)$. Let $\beta=\beta(v)\in \C^{(L+M)\times 2(L+M)}$ be such that $\beta \beta^{*}=I$ and $\beta J \beta^{*}=0$. Write $\beta=(\beta_1 \: \beta_2)$ with $\beta_i \in \C^{(L+M)\times (L+M)}$. For convenience, define $r_i=1$ for $i>\tilde k$. An interface condition at $v$ is imposed by requiring that 
\[
\beta_1 \begin{pmatrix} y_{i_1}'(-r_{i_1}) \\ \vdots \\ y_{i_L}'(-r_{i_L}) \\ -y_{j_1}'(r_{j_1}) \\ \vdots \\ -y_{j_M}'(r_{j_M}) \end{pmatrix} +
\beta_2 \begin{pmatrix} y_{i_1}(-r_{i_1}) \\ \vdots \\ y_{i_L}(-r_{i_L}) \\ y_{j_1}(r_{j_1}) \\ \vdots \\ y_{j_M}(r_{j_M}) \end{pmatrix}= 0
\]
for solutions $y_{i_p}$ and $y_{j_q}$ of $-y''+V_{i_p}y=zy$ and $-y''+V_{j_q}y=zy$, respectively. 

The domain $D(S)$ of the Schr\"odinger operator on $G$ is defined as the set of all $f \in \left(\bigoplus_{i=1}^{\tilde k} L^2(-r_i,r_i)\right) \oplus \left(\bigoplus_{i=\tilde k+1}^{k} L^2(-1,\infty)\right)$ such that $f_i$ and $f_i'$ are absolutely continuous, $-f_{i}''+V_if_i$ is square integrable, and
\begin{equation}
\label{Schric}
\beta_1 \begin{pmatrix} f_{i_1}'(-r_{i_1}) \\ \vdots \\ f_{i_L}'(-r_{i_L}) \\ -f_{j_1}'(r_{j_1}) \\ \vdots \\ -f_{j_M}'(r_{j_M}) \end{pmatrix} +
\beta_2 \begin{pmatrix} f_{i_1}(-r_{i_1}) \\ \vdots \\ f_{i_L}(-r_{i_L}) \\ f_{j_1}(r_{j_1}) \\ \vdots \\ f_{j_M}(r_{j_M}) \end{pmatrix}= 0 
\end{equation}
at every vertex. For $f\in D(S)$, define $Sf=\left(-f_{i}''+V_if_i \right)_{i=1}^{k}$. Then $S$ is a self-adjoint operator in $\left(\bigoplus_{i=1}^{\tilde k} L^2(-r_i,r_i)\right) \oplus \left(\bigoplus_{i=\tilde k+1}^{k} L^2(-1,\infty)\right)$. This is well-known, but it also follows from Theorem \ref{GSchrtoGCS} below.

The Schr\"odinger operator $S$ on $G$ is unitarily equivalent to a canonical system on $G$. The transformation of the equations along the edges is well-known. The only problem is to figure out what happens with the interface conditions and to prove everything in terms of the operators.

To see what happens to the equations inside the edges, let $T_i=\begin{pmatrix} p_i' & q_i' \\ p_i & q_i \end{pmatrix}$, where $p_i$ and $q_i$ are the solutions of $-y''+V_iy=0$ such that $p_i'(-r_i)=1=q_i(-r_i)$ and $q_i'(-r_i)=0=p_i(-r_i)$. If $-f''+V_if=g$, then $\tilde f=T_i^{-1}\begin{pmatrix} f' \\ f \end{pmatrix}$ and $\tilde g=T_i^{-1}\begin{pmatrix} 0 \\ g \end{pmatrix}$ solve $J\tilde f'=-H_i\tilde g$, where 
\begin{equation}
\label{CSH}
H_i=\begin{pmatrix} p_i^2 & p_iq_i \\ p_iq_i & q_i^2 \end{pmatrix} .
\end{equation}
If $J\tilde f'=-H_i\tilde g$, then $f=p_i\tilde f_1+q_i\tilde f_2$ and $g=p_i\tilde g_1+q_i\tilde g_2$ solve $-f''+V_if=g$. Note that $T_i \in \mathrm{SL}(2,\mathbb{R})$. The transfer matrices $T_i$ transform the interface conditions.

\begin{Lemma}
\label{iclemma}
Let $C\in \C^{2n\times 2n}$ be the matrix defined by
\[
Ce_i=\begin{cases} e_{(i+1)/2} & i \textrm{ odd} \\ e_{n+\frac{i}{2}} & i \textrm{ even} \end{cases} ,
\]
and let $A_i \in \mathrm{SL}(2,\mathbb{R})$ for $1\le i \le n$. If $\alpha \in \C^{n\times 2n}$ has rank $n$ and $\alpha J \alpha^*=0$, then $\tilde \alpha = \alpha C\left(\bigoplus_{i=1}^n A_i\right)C^*$ has the same two properties.
\end{Lemma} 
\begin{proof}
The claim about the rank is obvious because $C$ and $\bigoplus_{i=1}^n A_i$ are invertible. Now, $\tilde\alpha J \tilde\alpha^*=\alpha C\left(\bigoplus_{i=1}^n A_i\right)C^* JC\left(\bigoplus_{i=1}^n A_i\right)^*C^*\alpha^*$, and $C^*JC=\bigoplus_{i=1}^n J_1$, where $J_1=\begin{pmatrix} 0 & -1 \\ 1 & 0 \end{pmatrix}$. A simple calculation, using the assumption $A_i \in \mathrm{SL}(2,\mathbb{R})$, shows that $\left(\bigoplus_{i=1}^n A_i\right) \left(\bigoplus_{i=1}^n J_1\right) \left(\bigoplus_{i=1}^n A_i\right)^*=\bigoplus_{i=1}^n J_1$. So, $\tilde\alpha J \tilde\alpha^*=\alpha C\left(\bigoplus_{i=1}^n J_1\right)C^*\alpha^*=\alpha J\alpha^*=0$.   
\end{proof}

Define 
\begin{equation}
\label{SchrCSic}
\tilde\alpha=\beta C\left(\bigoplus_{p=1}^L T_{i_p}(-r_{i_p}) \oplus \bigoplus_{p=1}^M T_{j_p}(r_{j_p}) \right)C^* .
\end{equation}
and write $\tilde \alpha=(\tilde \alpha_1 \: \tilde \alpha_2)$ with $\tilde \alpha_i \in \C^{(L+M)\times (L+M)}$. Then, for any collection of absolutely continuous function $(f_i)$ on the edges, $(f_i)$ satisfies the interface conditions $\beta$ \eqref{Schric} at every vertex if and only if $(\tilde f_{(i)})=\left(T_i^{-1}\begin{pmatrix} f_i' \\ f_i \end{pmatrix} \right)$ satisfies
\[
\tilde\alpha_1 \begin{pmatrix} \tilde f_{(i_1)1}(-r_{i_1}) \\ \vdots \\ \tilde f_{(i_L)1}(-r_{i_L}) \\ -\tilde f_{(j_1)1}(r_{j_1}) \\ \vdots \\ -\tilde f_{(j_M)1}(r_{j_M}) \end{pmatrix} +
\tilde\alpha_2 \begin{pmatrix} \tilde f_{(i_1)2}(-r_{i_1}) \\ \vdots \\ \tilde f_{(i_L)2}(-r_{i_L}) \\ \tilde f_{(j_1)2}(r_{j_1}) \\ \vdots \\ \tilde f_{(j_M)2}(r_{j_M}) \end{pmatrix}= 0
\]
at every vertex. Since $\tilde\alpha$ has maximal rank $n$ and $\tilde\alpha J \tilde\alpha^*=0$ by Lemma \ref{iclemma}, one can replace $\tilde \alpha$ with a matrix $\alpha$ such that $\alpha$ induces the same interface condition, $\alpha \alpha^*=I$, and $\alpha J \alpha^*=0$.

Consider the canonical systems $H_i=\begin{pmatrix} p_i^2 & p_iq_i \\ p_iq_i & q_i^2 \end{pmatrix}$ with the interface conditions $\alpha$ on $G$. Note that any $H_i$ on a half line satisfies $H_i \notin L^1(-1,\infty)$ due to the limit point assumption on $V_i$. So, the relation $\mathcal S_{HG}$ corresponding to the canonical systems $H_i$ on $G$ with the interface conditions $\alpha$ is self-adjoint.

We now define a map $U$ that will provide the unitary equivalence between the Schr\"odinger operator $S$ on the graph and the self-adjoint relation $\mathcal S_{HG}$ corresponding to the canonical systems on the graph. Let $\mathcal H_1=\left(\bigoplus_{i=1}^{\tilde k} L^2(-r_i,r_i)\right) \oplus \left(\bigoplus_{i=\tilde k+1}^{k} L^2(-1,\infty)\right)$ and $\mathcal H_2=\left(\bigoplus_{i=1}^{\tilde k} L_{H_i}^2(-r_i,r_i)\right) \oplus \left(\bigoplus_{i=\tilde k+1}^{k} L_{H_i}^2(-1,\infty)\right)$. Introduce a map $U$ on $\mathcal H_1$ by setting
\begin{equation}
\label{USchrtoCS}
U\left(f_i\right)=\left(T_i^{-1}\begin{pmatrix} 0 \\ f_i \end{pmatrix}\right) .
\end{equation}
Then $U\left(f_i\right)$ is in $\mathcal H_2$ and $U:\mathcal H_1 \to \mathcal H_2$ is an isometry since  $T_i^{-1*}H_iT_i^{-1}=\begin{pmatrix} 0 & 0 \\ 0 & 1 \end{pmatrix}$ by direct calculation and the identity $p_i'q_i-p_iq_i'=1$. 

Next, we check that $U$ is surjective. Let $(\tilde f_{(i)})\in \mathcal H_2$ be given. Consider $f_i=p_i\tilde f_{(i)1}+q_i\tilde f_{(i)2}$. Note that $|f_i(x)|^2=\tilde f_{(i)}(x)^*H_i(x)\tilde f_{(i)}(x)$. So, $(f_i) \in \mathcal H_1$. Now, $U(f_i)$ and $(\tilde f_{(i)})$ are the same Hilbert space element since 
\begin{align}
\label{UfHil}
H_iT_i^{-1}\begin{pmatrix} 0 \\ f_i \end{pmatrix} &=\begin{pmatrix} p_i^2 & p_iq_i \\ p_iq_i & q_i^2 \end{pmatrix}\begin{pmatrix} -q_i'p_i\tilde f_{(i)1}-q_i'q_i\tilde f_{(i)2} \\ p_i'p_i\tilde f_{(i)1}+p_i'q_i\tilde f_{(i)2} \end{pmatrix}\\  
&=\begin{pmatrix} p_i^2\tilde f_{(i)1}\left(p_i'q_i-p_iq_i' \right)+p_iq_i\tilde f_{(i)2}\left(p_i'q_i-p_iq_i' \right) \\ p_iq_i\tilde f_{(i)1}\left(p_i'q_i-p_iq_i' \right)+q_i^2\tilde f_{(i)2}\left(p_i'q_i-p_iq_i' \right) \end{pmatrix} \nonumber \\
&=\begin{pmatrix} p_i^2 & p_iq_i \\ p_iq_i & q_i^2 \end{pmatrix}\tilde f_{(i)} \nonumber .
\end{align}
Thus, $U$ is surjective.

To check that $US \subset \mathcal S_{HG}U$, suppose $(f_i) \in D(S)$ is given, and let $(\tilde g_{(i)})=U(Sf_i)$. Then, as observed earlier, $\tilde f_{(i)}= T_i^{-1}\begin{pmatrix} f_i' \\ f_i \end{pmatrix}$ and $\tilde g_i$ solve $J\tilde f_{(i)}'=-H_i\tilde g_{(i)}$ and the functions $(\tilde f_{(i)})$ satisfy the interface conditions $\alpha$. Since $H_iT_i^{-1}=\begin{pmatrix} 0 & p_i \\ 0 & q_i \end{pmatrix}$, $(\tilde f_{(i)})$ and $U(f_i)$ are the same Hilbert space element. So, since $\left( \left(\tilde f_{(i)}\right), U\left(Sf_i\right)\right) \in \mathcal S_{HG}$, $\left( U\left(f_i\right), U\left(Sf_i\right)\right) \in \mathcal S_{HG}$.

Finally, to verify that $\mathcal S_{HG}U \subset US$, suppose $\left( U\left(f_i\right), \left(\tilde g_{(i)}\right)\right) \in \mathcal S_{HG}$ is given. So, there exist absolutely continuous functions $\left(\tilde f_{(i)}\right)$ that satisfy the interface conditions $\alpha$ and solve $J\tilde f_{(i)}=-H_i\tilde g_{(i)}$. As noted above, $f_i=p_i\tilde f_{(i)1}+q_i\tilde f_{(i)2}$ and $g_i=p_i\tilde g_{(i)1}+q_i\tilde g_{(i)2}$ solve $-f_i''+V_if_i=g_i$. A calculation, using the identity $p_i'q_i-p_iq_i'=1$ and the equation $J\tilde f_{(i)}=-H_i\tilde g_{(i)}$, shows that $\begin{pmatrix} f_i' \\ f_i\end{pmatrix}=T_i \begin{pmatrix} \tilde f_{(i)1} \\ \tilde f_{(i)2} \end{pmatrix}$. Hence, the functions $\left( f_i\right)$ satisfy the interface conditions $\beta$. Since $|f_i(x)|^2=\tilde f_{(i)}(x)^*H_i(x)\tilde f_{(i)}(x)$ and similarly for $\left( g_i\right)$, $\left( f_i\right)$ and $\left( g_i\right)$ are in $\mathcal H_1$. So, $S\left( f_i\right)=\left( g_i\right)$. By \eqref{UfHil} (replacing $f$ with $g$), $H_iT_i^{-1}\begin{pmatrix} 0 \\ g_i \end{pmatrix}=H_i\tilde g_{(i)}$. Thus, $\left(\tilde g_{(i)}\right)=U\left( g_i\right)=US\left( f_i\right)$. This completes the proof of the following theorem.
\begin{Theorem}
\label{GSchrtoGCS}
A Schr\"odinger operator $S$ on a graph is unitarily equivalent to a canonical system $\mathcal S_{HG}$ on the same graph, $USU^*=\mathcal S_{HG}$. The canonical systems $H_i$, the interface conditions, and $U$ are given by \eqref{CSH}, \eqref{SchrCSic}, and \eqref{USchrtoCS}, respectively.
\end{Theorem}
\begin{remark}
By combining this with the theorems in the preceding sections, one obtains the resolvent and a spectral representation of $S$.
\end{remark}
\begin{remark}
Similar transformations allow one to turn other kinds of problems on a graph, such as Dirac operators, into equivalent canonical systems on the graph. The transformation inside the edges is already known, and the interface conditions get transformed by $\mathrm{SL}(2,\mathbb{R})$ matrices as in Lemma \ref{iclemma} and \eqref{SchrCSic}.
\end{remark} 

\subsection*{Acknowledgements}
This research was supported by the European Union project CZ.02.1.01/0.0/0.0/16\textunderscore 019/0000778.


\begin{thebibliography}{100}
\bibitem{AP} S.\ Albeverio and K.\ Pankrashkin, A remark on Krein's resolvent formula and boundary conditions,
\textit{J.\ Phys.\ A }\textbf{38 }(2005), 4859--4865. 
\bibitem{ASVCdL} F.\ Andrade, A.G.M.\ Schmidt, E. Vincentini, B.K.\ Cheng, and M.G.E.\ da Luz, Green's function approach for quantum graphs,
\textit{Phys.\ Rep.\ }\textbf{647 }(2016), 1--46.
\bibitem{AK} S.\ Avdonin and P.\ Kurasov, Inverse problems for quantum trees,
\textit{Inverse Probl.\ Imaging }\textbf{2 }(2008), 1--21.  
\bibitem{BHS} J.\ Behrndt, S.\ Hassi, and H.\ de Snoo, Boundary Value Problems, Weyl Functions, and Differential Operators,
Monographs in Mathematics, 108, Birkh\"auser, Cham, 2020.
\bibitem{BHSW} J.\ Behrndt, S.\ Hassi, H.\ de Snoo, and R.\ Wietsma, Square-integrable solutions and Weyl functions for singular canonical systems,
\textit{Math.\ Nachr.\ }\textbf{284 }(2011), 1334--1384.
\bibitem{BL} J.\ Behrndt and A.\ Luger, On the number of negative eigenvalues of the Laplacian on a metric graph,
\textit{J.\ Phys.\ A }\textbf{47 }(2010), 474006. 
\bibitem{BCFK} G.\ Berkolaiko, R.\ Carlson, S.\ Fulling, and P.\ Kuchment, editors, Quantum Graphs and Their Applications,
Contemporary Mathematics, 415, American Mathematical Society, Providence, 2006.
\bibitem{BK} G.\ Berkolaiko and P.\ Kuchment, Introduction to Quantum Graphs,
Mathematical Surveys and Monographs, 186, American Mathematical Society, Providence, 2013.
\bibitem{BER} J.\ Bolte, S.\ Egger, and R.\ Rueckriemen, Heat-kernel and resolvent asymptotics for Schr\"odinger operators on metric graphs,
\textit{Appl.\ Math.\ Res.\ Express }\textbf{2015 }(2015), 129-–165.
\bibitem{BLT} B.M.\ Brown, H.\ Langer, and C.\ Tretter, Compressed resolvents and reduction of spectral problems on star graphs,
\textit{Complex Anal.\ Oper.\ Theory }\textbf{13 }(2019), 291--320.
\bibitem{BW} B.M.\ Brown and R.\ Weikard, A Borg-Levinson theorem for trees,
\textit{Proc.\ R.\ Soc.\ Lond.\ Ser.\ A Math.\ Phys.\ Eng.\ Sci.\ }\textbf{461 }(2005), 3231--3243.
\bibitem{CW1} S.\ Currie and B.\ Watson, Eigenvalue asymptotics for differential operators on graphs,
\textit{J.\ Comput.\ Appl.\ Math.\ }\textbf{182 }(2005), 13--31. 
\bibitem{CW2} S.\ Currie and B.\ Watson, $M$-matrix asymptotics for Sturm-Liouville problems on graphs,
\textit{J.\ Comput.\ Appl.\ Math.\ }\textbf{182 }(2008), 568--578. 
\bibitem{CW3} S.\ Currie and B.\ Watson, The $M$-matrix inverse problem for the Sturm-Liouville equation on graphs,
\textit{Proc.\ Roy.\ Soc.\ Edinburgh Sect.\ A }\textbf{139 }(2009), 775--796. 
\bibitem{dB} L.\ de Branges, Some Hilbert spaces of entire functions II,
\textit{Trans.\ Amer.\ Math.\ Soc.\ }\textbf{99 }(1961), 118--152.
\bibitem{dSW} H.\ de Snoo and H.\ Winkler, Canonical systems of differential equations with self-adjoint interface conditions on graphs,
\textit{Proc.\ Roy.\ Soc.\ Edinburgh Sect.\ A }\textbf{135 }(2005), 297--315.
\bibitem{EKKST} P.\ Exner, J.P.\ Keating, P. Kuchment, T.\ Sunada, and A.\ Teplyaev, editors, Analysis on Graphs and Its Applications,
Proceedings of Symposia in Pure Mathematics, 77, American Mathematical Society, Providence, 2008.
\bibitem{GT} F.\ Gesztesy and E.\  Tsekanovskii, On matrix-valued Herglotz functions,
\textit{Math.\ Nachr.\ }\textbf{218 }(2000), 61--138.
\bibitem{HS1} D.B.\ Hinton and J.K.\ Shaw, On Titchmarsh-Weyl $M(\lambda)$ functions for linear Hamiltonian systems,
\textit{J.\ Differential Equations }\textbf{40 }(1981), 316--342.
\bibitem{HS2} D.B.\ Hinton and J.K.\ Shaw, On the spectrum of a singular Hamiltonian system,
\textit{Quaestiones Math.\ }\textbf{5 }(1982), 29--81.
\bibitem{KS} V.\ Kostrykin and R.\ Schrader, Kirchhoff's rule for quantum wires,
\textit{J.\ Phys.\ A }\textbf{32 }(1999), 595--630.
\bibitem{Krall} A.\ Krall, $M(\lambda)$ theory for singular Hamiltonian systems with one singular point,
\textit{SIAM J.\ Math.\ Anal.\ }\textbf{20 }(1989), 664--700.
\bibitem{KN} P.\ Kurasov and S.\ Naboko, Gluing graphs and the spectral gap: a Titchmarsh-Weyl matrix-valued function approach,
\textit{Studia Math.\ }\textbf{255 }(2020), 303--326.
\bibitem{LSS} D.\ Lenz, C.\ Schubert, and P.\ Stollmann, Eigenfunction expansions for Schr\"odinger operators on metric graphs,
\textit{Integral Equations Operator Theory }\textbf{62 }(2008), 541-–553.
\bibitem{Lesch} M.\ Lesch and M.\ Malamud, On the deficiency indices and self-adjointness of symmetric Hamiltonian systems,
\textit{J.\ Differential Equations }\textbf{189 }(2003), 556--615.
\bibitem{Post} O.\ Post, Spectral Analysis on Graph-Like Spaces,
Lecture Notes in Mathematics, 2039, Springer, Heidelberg, 2012..
\bibitem{Rembook} C.\ Remling, Spectral Theory of Canonical Systems, de~Gruyter Studies in Mathematics, 70,
Berlin/Boston, 2018.
\bibitem{Roh} J.\ Rohleder, Recovering a quantum graph spectrum from vertex data,
\textit{J.\ Phys.\ A }\textbf{48 }(2015), 165202.
\bibitem{Sakh} L.\ Sakhnovich, Spectral Theory of Canonical Differential Systems. Method of Operator Identities,
Operator Theory: Advances and Applications, 107, Birkh\"auser Verlag, Basel, 1999.
\bibitem{SW} S.\ Simonov and H.\ Woracek, Spectral multiplicity of selfadjoint Schr\"odinger operators on star-graphs with standard interface conditions,
\textit{Integral Equations Operator Theory }\textbf{78 }(2014), 523-–575.
\bibitem{Teschl} G.\ Teschl, Mathematical Methods in Quantum Mechanics,
Graduate Studies in Mathematics, 157, American Mathematical Society, Providence, 2014.
\bibitem{WMLN} J.\ Weidmann, Spectral Theory of Ordinary Differential Operators,
Springer Lecture Notes, 1258, Springer, Berlin, 1987.
\bibitem{Yur} V.\ Yurko, Inverse problems for differential operators of variable orders on star-type graphs: general case,
\textit{Anal.\ Math.\ Phys.\ }\textbf{4 }(2014), 247-–262.
\end{thebibliography}
\end{document}